\documentclass[1pt]{article}
\usepackage{amsmath, amssymb, amsthm, amsfonts, cases}
\usepackage{mathrsfs}
\usepackage{url}
\usepackage{authblk}
\usepackage[usenames]{color}
\usepackage{geometry}
\usepackage{stmaryrd}
\allowdisplaybreaks

\theoremstyle{plain}
\newtheorem{thm}{Theorem}[section]

\newtheorem{pro}[thm]{Problem}
\newtheorem{lem}[thm]{Lemma}
\newtheorem{prop}[thm]{Proposition}

\theoremstyle{definition}
\newtheorem{defn}{Definition}[section]
\newtheorem{ass}{Assumption}[section]
\newtheorem{rmk}{Remark}[section]

\newcommand{\la}{\langle}
\newcommand{\ra}{\rangle}

\makeatletter\@addtoreset{equation}{section} \makeatother

\DeclareMathOperator*{\essinf}{ess\,inf}

\begin{document}
\title{Optimal controls of stochastic differential equations with jumps and random coefficients: Stochastic Hamilton-Jacobi-Bellman equations with jumps
\thanks{Q. Meng was partially supported by the National Natural Science Foundation of China (No. 11871121), and by the Natural Science Foundation of Zhejiang Province
	(No. LY21A010001). Y. Shen  was partially supported by the Australian Research Council (No. DE200101266). S. Tang  was supported by National Science Foundation of China (Grant No. 11631004) and National
	Key R\&D Program of China (Grant No. 2018YFA0703903)}}

\date{}
\author[a]{Qingxin Meng}
\affil[a]{\small  Department of Mathematical Sciences, Huzhou University, Zhejiang 313000, China }

\author[b]{Yuchao Dong,\footnote{Corresponding author. Email: dyc19881021@icloud.com}}
\affil[b]{\small{Institution of Operations Research and Analytics, National University of Singapore}}

\author[c]{Yang Shen}
\affil[c]{\small{School of Risk and Actuarial Studies and CEPAR,
University of New South Wales Sydney, NSW 2052, Australia}}

\author[d]{Shanjian Tang}
\affil[d]{\small Department of Finance and Control Sciences, School of Mathematical Sciences, Fudan University, Shanghai 200433, China}

\maketitle

\begin{abstract} In this paper, we study the following
nonlinear backward stochastic integral partial differential equation with  jumps
\begin{equation*}
\left\{
\begin{split} -d V(t,x) =&\displaystyle\inf_{u\in U}\bigg\{H(t,x,u, DV(t,x),D \Phi(t,x), D^2 V(t,x),\int_E \left(\mathcal I V(t,e,x,u)+\Psi(t,x+g(t,e,x,u))\right)l(t,e)\nu(de)) \\
&+\displaystyle\int_{E}\big[\mathcal I V(t,e,x,u)-\displaystyle
(g(t, e,x,u), D V(t,x))\big]\nu(d e)+\int_{E}\big[\mathcal I \Psi(t,e,x,u)\big]\nu(d e)\bigg\}dt\\
&-\Phi(t,x)dW(t)-\displaystyle\int_{E} \Psi (t, e,x)\tilde\mu(d e,dt),\\
V(T,x)=& \ h(x),
\end{split}
\right.
\end{equation*}
where $\tilde \mu$ is a Poisson random martingale measure, $W$ is a Brownian motion, and $\mathcal I$ is a non-local operator to be specified later. The function $H$ is a given random mapping, which arises from a corresponding non-Markovian optimal control problem. This equation appears as the
stochastic Hamilton-Jacobi-Bellman equation, which characterizes the value function of the optimal control problem with a recursive utility cost functional. The solution to the equation is a predictable triplet of
random fields $(V,\Phi,\Psi)$. We show that the value function, under some regularity assumptions, is the solution to the stochastic HJB equation; and a classical solution to this equation is the value function and gives the optimal control. With some additional assumptions on the coefficients, an existence and uniqueness result in the
sense of Sobolev space is shown by recasting the backward stochastic
partial integral differential equation with jumps as a backward
stochastic evolution equation in Hilbert spaces with Poisson jumps.

\end{abstract}

\textbf{Keywords}: Stochastic control; Dynamic programming; Stochastic HJB equation;
Stochastic partial differential equation

\maketitle

\section{Introduction}

Backward
stochastic partial differential equations (BSPDEs) are  natural
generalization of backward stochastic differential equations (BSDEs).
The theory of BSDEs has been well developed, dating back to \cite{Bism76}
and \cite{PaPe90} on the linear and nonlinear cases, respectively.
A systematic account of the theory and application of BSDEs is available
in \cite{KPQ97}, \cite{ZhYo}, \cite{Zhang2017}, and the references
therein. In the jump cases, BSDEs have also been studied by many
authors (see, for example, \cite{tang1994necessary}, \cite{Situ97}, and \cite{FOS05}).
In recent years, there has been growing interest on BSPDEs, partly driven by
its wide variety of applications in stochastic
optimal control theory and mathematical finance. For instance, BSPDEs serve as adjoint equations in
Pontryagin's maximum principle when the controlled system is a
stochastic partial differential equations. For this line of research, one can refer to
\cite{Bens92}, \cite{NaNi90}, \cite{HuPe91}, \cite{Zhou93},
\cite{MaYo97}, \cite{Tang98a}, and \cite{Tang05}. Recently, \cite{BQY2020} studied the option pricing
problems under rough volatility models and proved that the value function satisfies a BSPDE.

The backward stochastic Hamilton-Jacobi-Bellman (HJB) equations,
which is a class of fully nonlinear BSPDEs, was first introduced by Peng \cite{Peng92a}
in the study of stochastic control systems driven by Brownian motions, where the
coefficients of the control systems are allowed to be random.
The stochastic HJB equations have a stochastic control interpretation as
the classical deterministic HJB equation of which the classical
solution is the value function of the stochastic optimal control
problem with deterministic coefficients. In \cite{Peng92a},
for the case  where the diffusion term of the system does not
contain the control variables, the existence and uniqueness of
adapted solutions to the stochastic HJB equations driven by Brownian
motions in the sense of Sobolev space were established by regarding
the equations as a class of backward stochastic evolution equations
for Hilbert space valued processes. However, in general the existence and
uniqueness results of the classical solution for the stochastic HJB
equations have been still an open problem.

The purpose of this paper is to study a stochastic HJB equation driven by both Brownian motions and Poisson jumps.
More specifically, the stochastic HJB equation studied in this paper
is associated with stochastic optimal problems driven
by Brownian motions and Poisson jumps simultaneously and with the recursive utility type cost functional,
which is given by a BSDE. In economics and finance, recursive utilities have been used to disentangle the investor's risk aversion and intertemporal substitution, see \cite{epstein1991substitution}.
See also \cite{Duffie1992} for a stochastic differential formulation of recursive utility, namely stochastic differential utility. Using the BSDE theory, the stochastic differential utility was extended to the case with multiple priors
that distinguish risk and ambiguity in a unified framework, see \cite{chen2002ambiguity}.
Using dynamic programming principle, we first show that the value function, under some regularity assumptions, is the solution to the stochastic HJB equation.
However, this result does not apply to a general case since in general the value function does not satisfy these regularity assumptions. On the other hand, we provide a verification theorem and show that if the stochastic HJB equation admits a classical solution which is a triplet of random fields and satisfies sufficient regularity conditions, then the first component of the triplet coincides with the value function of the optimal control problem. Prior to our work, Li and Peng \cite{li2009stochastic} also studied the optimal control problem with jumps. Unlike our framework accommodating random coefficients, they focused on the case with deterministic coefficients, thereby proving the relation between value function and the viscosity solution of (deterministic) HJB equation. Allowing the coefficients to be random results in
the main difficulty of our study. For example, in \cite{li2009stochastic}, the authors showed that the value function is continuous in $t$. Unfortunately, this is not the case in our framework. Moreover, the viscosity solution of BSPDE is not well-studied. To our best knowledge, most works considered the control systems only driven by Brownian motions, see \cite{qiu2018viscosity} and the references therein.

Instead of considering the viscosity solution for the most general case, we study the stochastic HJB equation with jumps in the sense of Sobolev spaces
and prove the existence and uniqueness results in these spaces with some additional assumptions. To this end, the most important step
is that we first establish the existence and uniqueness results for
a class of backward stochastic evolution equations with jumps in
Hilbert spaces. Secondly, we recast the stochastic HJB equation with
jumps as the backward stochastic evolution equation with jumps in
Hilbert spaces, where the existence and uniqueness results are applicable. It is worth noting that {\O}ksendal et al. \cite {ORZ05}
investigated a class of semi-linear backward stochastic partial
differential equations with jumps which appear as adjoint equations
in the maximum principle for optimal control of stochastic partial
differential equations driven by Poisson jumps. Our formulation here is more general than theirs.
Thus, the results in \cite{ORZ05} can be covered by ours as special cases.

The rest of  this paper is organized as follows. In Section 2, we introduce some notations and preliminary results. The dynamic programming principle is given in Section 3. In Sections 4, we
provide the optimal control interpretation of the stochastic HJB
equations with jumps and a related verification theorem. Section 5
is devoted to studying the existence and uniqueness theory of
adapted solutions to backward stochastic evolution equations with random jumps. With this result,
we prove the existence and uniqueness results for a special class of the stochastic HJB equations with jumps.

\section{Notations, preliminary and problem formulation}


In this section, we first introduce basic notations, standing assumptions, and a preliminary result on
the \emph{essential infimum} of a family of nonnegative random variables. Then, we formulate the stochastic control
problem with jumps and random coefficients.

\subsection{Notations and Preliminary}

Let ${\cal T} : = [0, T]$ denote a fixed time interval of finite length, i.e., $T < \infty$.
We consider a complete probability space $( \Omega, {\mathscr F}, {\mathbb P} )$,
on which all randomness are defined. The space $( \Omega, {\mathscr F}, {\mathbb P} )$
is equipped with a right-continuous, ${\mathbb P}$-complete filtration ${\mathbb F} : = \{ {\mathscr F}_t | t \in {\cal T} \}$,
to be specified later. Furthermore, we assume that ${\mathscr F} _{T} = {\mathscr F}$. Denote by
${\mathbb E} [\cdot]$ the expectation with respect to ${\mathbb P}$, by $\mathscr{P}$
the predictable $\sigma$-algebra on $\Omega\times {\cal T} $ associated with $\mathbb F$,
and by $\mathscr B(\Lambda)$ the Borel $\sigma$-algebra of any topological space $\Lambda$.
Let $\{ W (t) | t \in {\cal T} \}
\triangleq \{ (W_1 (t), W_2 (t), \cdots, W_d (t))^{\top} | t \in \cal T \}$ be a $d$-dimensional
standard Brownian motion  under the
 probability $( \Omega, {\mathscr F}, {\mathbb P} )$, $(E,\mathscr B (E), \nu
)$ be a measure space with $\nu(E)<\infty$, and $\eta$ be a
stationary Poisson point process with the characteristic measure $\nu$ (see \cite{ikeda2014stochastic} for details). Then, the counting measure induced by $\eta$ is
$$
\mu((0,t]\times A)\triangleq\#\{s; s\leq t, \eta(s)\in A\},~~\mbox{for}~ t>0, ~ A\in \mathscr B (E),
$$
and $\tilde{\mu}(de,dt)\triangleq\mu(de,dt)-\nu(de)dt$ is a compensated Poisson random martingale  measure which
is assumed to be independent of the Brownian motion $\{ W (t) | t \in {\cal T} \}$.
Moreover, the filtration $\mathbb F$ is the $\mathbb P$-augmentation of the natural filtration generated by the Brownian motion
$\{ W (t) | t \in {\cal T} \}$ and the Poisson random measure $\{\mu((0,t]\times A)| t \in {\cal T},  A\in \mathscr B (E)\}$. Let $T_0$
be the set of all stopping times 
bounded by $T$. For any stopping time
$\tau\in T_0$, denote  by $T_\tau$ the set of all stopping times in $T_0$ and greater than $\tau$.

Let $H$ be any Hilbert space. The inner product in $H$ is denoted by
$(\cdot, \cdot),$ and the norm in $H$ is denoted by $||\cdot||_H$ or $|\cdot|$ whenever there is no risk of confusion. For a scalar-valued
function $\phi:\mathbb R^n \rightarrow \mathbb R$, we denote by
$\phi_x$, $D_x \phi$, or $D\phi$ its gradient and $\phi_{xx}$, $D_{x}^2 \phi$, or $D^2\phi$ its Hessian, which is a symmetric
matrix. For a vector-valued function $\phi: \mathbb R^n\rightarrow \mathbb R^k$ (with
$k\geq 2)$, $\phi_x=(\frac{\partial \phi_i}{\partial x_j})$ is
the corresponding $(k\times n)$ Jacobian matrix. For any two stopping times $\tau$ and $\gamma$,
the corresponding stochastic interval is defined by the following set
\begin{eqnarray*}
\llbracket\tau,\gamma\rrbracket \triangleq \{ (t,\omega) \in [0,\infty) \times \Omega| \tau (\omega) \leq t \leq \gamma (\omega) \} .
\end{eqnarray*}

The following spaces will be frequently used in this paper: for any $\tau\in T_0$ and $\gamma\in T_\tau$,
\begin{enumerate}
\item[$\bullet$] $ M_{\mathscr{F}}^{2,p}(\tau,\gamma;H):$ the space of all $H$-valued
and ${\mathscr{F}}_t$-predictable processes $f=\{f(t,\omega)|
(t,\omega)\in \llbracket \tau,\gamma \rrbracket \times\Omega\}$ satisfying
$\| f \|^p_{{M}_{\mathscr F}^{2,p} (\tau,\gamma; H)} \triangleq
{\mathbb E} \big [ (\int_\tau^\gamma | f(t) |^2 d t)^p \big ]  < \infty$;

\item[$\bullet$] $M_{\mathscr{F}}^{2}(\tau,\gamma;H):$ the space of all $H$-valued
and ${\mathscr{F}}_t$-predictable processes $f=\{f(t,\omega)|
(t,\omega)\in \llbracket \tau,\gamma \rrbracket \times\Omega\}$ satisfying
$\| f \|^p_{{M}_{\mathscr F}^{2,p} (\tau,\gamma; H)} \triangleq
{\mathbb E} \big [ \int_\tau^\gamma | f(t) |^2 d t \big ]  < \infty$, i.e., $M_{\mathscr{F}}^{2}(\tau,\gamma;H)
= M_{\mathscr{F}}^{2,1}(\tau,\gamma;H)$;

\item[$\bullet$] ${S}_{\mathscr F}^2 (\tau,\gamma; H)$: the space of all $H$-valued
and ${\mathscr{F}}_t$-predictable  c\`{a}dl\`{a}g processes
$f=\{f(t,\omega)| (t,\omega)\in\llbracket\tau,\gamma\rrbracket\times\Omega\}$ satisfying
$\| f \|_{{ S}_{\mathscr F}^2 (\tau,\gamma; H)} \triangleq \sqrt { {\mathbb E} \big [
\sup_{\tau \leq t\leq \gamma} | f(t) |^2 \big ] } < \infty$;

\item[$\bullet$] ${M}^{\nu,2}( E; H):$ the space of all $H$-valued measurable
  functions $r=\{r(e)| e \in E\}$ defined on the measure space $(E, \mathscr B(E); \nu)$ satisfying
$$\|r\|_{{ M}^{\nu,2}( E; H)}^2\triangleq{\displaystyle\int_E|r(\theta)|^2\nu(d\theta)}<\infty;$$

\item[$\bullet$] ${M}_{\mathscr{F}}^{\nu,2}{(\llbracket\tau,\gamma\rrbracket\times  E; H)}:$ the  space of all ${M}^{\nu,2}( E; H)$-valued
and ${\mathscr F}_t$-predictable processes $r=\{r(t,\omega,e)|
(t,\omega,e)\in\llbracket\tau,\gamma\rrbracket\times\Omega\times E\}$ satisfying
$$\|r\|_{M_{\mathscr F}^{v,2}(\llbracket\tau,\gamma\rrbracket\times  E; H)}^2\triangleq {\mathbb E\bigg[\int_\tau^\gamma\displaystyle\|r(t,\cdot)\|^2_
{{M}^{\nu,2}( E; H)}dt}\bigg]<\infty;$$

\item[$\bullet$] $L^2(\Omega,{\mathscr G},\mathbb P;H):$ the space of all $H$-valued, ${\mathscr G}$-measurable,
random variables $\xi$ defined on $(\Omega,{\mathscr G},P)$ satisfying  $\| \xi \|_{L^2(\Omega,{\mathscr G},P;H)} \triangleq\sqrt{\mathbb E|\xi|^2}$,
where $\mathscr G$ is a sub-algebra of $\mathscr F$.

\end{enumerate}

In what follows, we recall a classical theorem for the \emph{essential infimum}
of a family  of nonnegative random variables in a probability space (see  \cite[Appendix A]{karatzas1998methods}).

\begin{lem}\label{lem:2.1}
Let $\mathscr X$ be a family of nonnegative integrable random variables defined on the probability space $(\Omega, \mathscr F, \mathbb P)$.
Then there exists a  random variable $X^*$ such that
\begin{enumerate}
\item for all $X\in \mathscr X,$ $X\geq X^*$ a.s.;
\item if $Y$ is another random variable satisfying $ X\geq Y $ a.s., for all $ X\in \mathscr X$, then $X^*\geq Y$ a.s..
\end{enumerate}
The random variable $X^*$, which is unique a.s., is called the essential infimum of $\mathscr X$, and is denoted by $\essinf\mathscr X$
or $\essinf\limits_{X\in \mathscr X}X$.

Furthermore, if $\mathscr X$ is closed  under pairwise minimum (i.e., $X, Y\in \mathscr X$ implies
$X\wedge Y\in \mathscr X$), then there exists a nonincreasing  sequence  $\{Z_n\}_{n\in \mathbb N}$
of random variables in $\mathscr X$ such that $X^*= \lim\limits_{n\rightarrow\infty} Z_n$ a.s.. Moreover, for any sub-algebra $\mathscr G$ of $\mathscr F$,
the $\mathscr G$-conditional expectation is interchangeable with the essential infimum, that is,
\begin{eqnarray}
\mathbb E \left [ \essinf\limits_{X\in \mathscr X}X \bigg |\mathscr G \right ]= \essinf\limits_{X\in \mathscr X} \mathbb E[ X \big|\mathscr G] \text{ a.s.}.
\end{eqnarray}
\end{lem}
\begin{rmk}\label{rmk_essinf}
The above result can be extended to the family $\mathscr X$ of random variables that are uniformly bounded from below by another random variable $Y$, i.e., $X \ge Y$, $\forall X \in \mathscr X$. For that purpose, we only need to apply Lemma \ref{lem:2.1} to the family $\{ X-Y| X \in \mathscr X\}$ to get the desired result.
\end{rmk}

\subsection{Statement of the control problem}

Let $U$ be a nonempty  subset of $\mathbb R^m.$ In this paper, the admissible control is defined as follows:

\begin{defn}
For any  given $t \in {\cal T}$, a stochastic process $u(\cdot)$ defined on
$[t, T]\times \Omega$ is said to be an admissible control on the interval $[t, T]$,
if $u(s)\in U$ for almost $(s,\omega)\in [t, T]
\times \Omega$  and $ u(\cdot)$ is an $\{\mathscr F_s | s \in [t, T] \}$-predictable processes. 
The set of all admissible controls on the interval $[t, T]$ is denoted by $\mathscr V [t, T]$.
\end{defn}

In this paper, for any given initial state $x \in {\mathbb R}^n$ and admissible control $ u(\cdot) \in \mathscr V [0, T]$,
we consider the following  controlled stochastic differential equation driven by the Brownian motion $W$ and the Poisson random martingale measure $\tilde \mu$:
\begin{eqnarray}\label{eq:b3}
\left\{
\begin{aligned}
dX(s) =& \ b(s,X(s),u(s))ds+\sigma(s,X(s),u(s))d W(s)+ \int_Eg(s, e, X({s-}), u(s))\tilde\mu(d e, ds), \quad 0\leq s\leq T,\\
X(0) =& \ x_0\in \mathbb R^n,
\end{aligned}
\right.
\end{eqnarray}
where the coefficients $b$, $\sigma$, and $g$ are given
random mappings satisfying the following assumption:
\begin{ass}\label{ass:2.1}
\begin{enumerate}
\item[(i)] The mappings $b:{\cal T}\times\Omega\times \mathbb R^n\times U\longrightarrow \mathbb R^n$ and
$\sigma:{\cal T}\times \Omega\times \mathbb R^n\times U\longrightarrow\mathbb R^{n\times d}$ are
$\mathscr P\otimes {\mathscr B} (\mathbb R^n) \otimes {\mathscr B} ({U})$-measurable;
the mapping $g:{\cal T}\times \Omega\times {E}\times \mathbb R^n\times U\longrightarrow \mathbb R^n$
is  $\mathscr P \otimes {\mathscr B} ({E})\otimes {\mathscr B} (\mathbb R^n) \otimes {\mathscr B} ({U})$-measurable.
\item[(ii)] There exists a positive constant $C$
and  deterministic nonnegative function $\rho(e)$  such that
for all   $(u, u',x,x') \in U\times U\times {\mathbb R}^n
\times {\mathbb R}^n $ and a.e. $(t,\omega,e) \in  {\cal T} \times \Omega
\times E$,
\begin{eqnarray*}
\left\{
\begin{aligned}
&|b(t,x,u)-b(t,x',u')|+|\sigma(t,x,u)-\sigma(t,x',u')|\leq C \big(|x-x'|+|u-u'|\big),
\\& |g(t,e,x,u)-g(t,e,x',u')|\leq \rho(e) \big(|x-x'|+|u-u'|\big),
\\&
|b(t,x,u)|+|\sigma(t,x,u)|\leq C\big(1+|x|+|u|\big),
\\&
|g(t,e,x,u)|\leq \rho(e)\big(1+|x|+|u|\big),
\end{aligned}
\right.
\end{eqnarray*}
and
\begin{eqnarray*}
  \int_{E}\exp{\{\rho(e)\}}\nu(de)<\infty.
\end{eqnarray*}
\end{enumerate}
\end{ass}
Furthermore, we impose the following assumption on the control region.
 \begin{ass}\label{ass:3.1}
	The control domain $U$ is a compact subset of
	$\mathbb R^m$.
\end{ass}
Under Assumptions \ref{ass:2.1}-\ref{ass:3.1}, for any initial value $X(0)=x_0\in \mathbb R^n$ and admissible control $u(\cdot)\in \mathscr V [0, T]$,
the SDE \eqref{eq:b3} admits a unique strong solution satisfying $X^{0,x_0;u}(\cdot)\in {S}^p_{\mathscr F} ( 0, T; \mathbb R^n )$, for any $p>1$.  The solution $X(\cdot) \triangleq X^{0,x_0; u}(\cdot) $  to the SDE \eqref{eq:b3}
 is referred to as the state process
corresponding to the admissible control process $u (\cdot)$ and the pair of stochastic processes $( u (\cdot); X (\cdot) )$ is referred to as an admissible pair.

For any admissible pair $( u (\cdot); X (\cdot) )$,  consider the following BSDE with jumps
\begin{equation}\label{eq:b5}
\begin{split}
Y(t)=h(X(T))&+\int_t^T f(s,X(s),u(s),Y(s),Z(s),\int_E K(s,e)l(s,e)\nu(de))ds\\
&-\int_t^T Z(s)dW(s)-\int_t^T\int_E K(s,e)\tilde \mu(de,ds), \quad 0\le t\le T,
\end{split}
\end{equation}
where $f$, $h$, and $l$ are given random mappings
satisfying the following assumption:
\begin{ass}\label{ass:2.2}
\begin{enumerate}
\item[(i)] $f: {\cal T} \times \Omega \times \mathbb
R^n \times {U}\times \mathbb R \times \mathbb R^{d}\times\mathbb R  \rightarrow {\mathbb R}$
is  ${\mathscr P}\otimes {\mathscr B} (\mathbb R^n) \otimes {\mathscr B} ({U}) \otimes \mathcal B(\mathbb R)\otimes \mathcal B(\mathbb R^{d})\otimes \mathcal B( \mathbb R)$-measurable; $l:\mathcal T\times \Omega\times E \rightarrow \mathbb R$ is $\mathscr P \times \mathcal B(E)$-measurable; $h: \Omega \times \mathbb R^n \rightarrow {\mathbb R}$
is ${\mathscr F}_T\otimes {\mathscr B} (\mathbb R^n) $-measurable.
\item[(ii)] There exists a positive constant $C$ such that
for almost all $(t, \omega) \in \mathcal T \times \Omega$ and $(x, u, y, z, k), (x', u', y', z', k') \in \mathbb R^n \times U \times \mathbb R \times \mathbb R^d \times \mathbb R$,
\begin{eqnarray*}
\left\{
\begin{aligned}
&|f(t,x,u,y,z,k)-f(t,x',u',y',z',k')|+|h(x)-h(x')|\\
&\leq C \left\{\big(1+|x|+|x'|+|u|+|u'|\big) \big(|x-x'|+|u-u'|\big)+|y-y'|+|z-z'|+|k-k'| \right\},
\\&
|f(t,x,u,y,z,k)|+|h(x)|\leq C\big(1+|x|^2+|u|^2+|y|+|z|+|k|\big).
\end{aligned}
\right.
\end{eqnarray*}
\item[(iii)]  $k \mapsto f(t,x,u,y,z,k)$ is non-decreasing for all $(t,x,u,y,z) \in {\cal T} \times \mathbb
R^n \times {U}\times \mathbb R \times \mathbb R^{d}$;
there exists a positive constant $C$ such that $0\le l(t,e) \le C(1+|e|)$ for all $(t,e) \in {\cal T} \times E$.
\end{enumerate}
\end{ass}
From Assumption \ref{ass:2.2}, we see that \eqref{eq:b5} admits a unique solution $(Y,Z,K) \in \mathcal S^2_{\mathscr F}(0,T;\mathbb R) \times  \mathcal M^{2,1}_{\mathscr F}(0,T;\mathbb R^d)\times \mathcal M^{\nu,2}_{\mathscr F}([0,T]\times E;\mathbb R)$. Moreover, Condition (iii) in Assumption \ref{ass:2.2} ensures that the comparison principle holds (see \cite{li2009stochastic}). We may also denote the process as $(Y^{0,x_0;u},Z^{0,x_0;u},K^{0,x_0;u})$ to emphasize the dependence on the initial data and admissible control whenever necessary. The cost functional is defined as
$$
J(0,x_0;u(\cdot))=Y(0).
$$
In this paper, we also need the following assumption for the coefficient on the jump part, which ensures related stochastic flows to be invertible.
\begin{ass}\label{ass:c4}
	The map $\phi_{t,e,u}:x \rightarrow x+g(t,e,x,u)$ is homeomorphic from $\mathbb R^n$ to $\mathbb R^n$ and the inverse map $\psi_{t,e,u}$ has uniformly linear growth and is uniformly Lipschitz continuous. Moreover, there exists a positive  constant $\delta$ such that
	\begin{eqnarray}
	\big|\det (I+D_x g(t, e,x,u))
	\big|\geq \delta, ~~~\forall (t, e,x,u)
	\in {\cal T}\times {E}\times \mathbb R^n\times U.
	\end{eqnarray}
\end{ass}

Under Assumptions \ref{ass:2.1}-\ref{ass:2.2},
it is easy to check that $$|J(0,x_0; u(\cdot))|
<\infty.$$ Thus, the cost functional
\eqref{eq:b5} is well-defined. We are now ready to state our optimal control problem:

\begin{pro}\label{pro:2.1}
Find an admissible control process ${\bar u} (\cdot) \in \mathscr V [0, T]$ such that
\begin{eqnarray} \label{eq:b8}
J (0,x_0;{\bar u}(\cdot) ) = \inf_{u (\cdot) \in \mathscr V [0, T]} J (0,x_0; u (\cdot) )
\end{eqnarray}
subject to \eqref{eq:b3} and \eqref{eq:b5}.
\end{pro}
The admissible control ${\bar u} (\cdot)\in\mathscr V [0, T]$ satisfying (\ref{eq:b8}) is called an optimal control process of
Problem \ref{pro:2.1}. Correspondingly, the state process ${\bar X} (\cdot)$ associated with ${\bar u} (\cdot)$
is called an optimal state process, and $( {\bar u} (\cdot); {\bar X} (\cdot) )$ is called an optimal pair of
Problem \ref{pro:2.1}.

\section{Bellman's dynamic programming principle with jumps}

One of the key features of Problem \ref{pro:2.1} is that all the coefficients in the state equation \eqref{eq:b3}
and the cost functional \eqref{eq:b5} are stochastic processes or random variables.
Therefore, Problem \ref{pro:2.1} is indeed a non-Markovian optimal stochastic control problem.

Two of the most important approaches to solving stochastic optimal control problems are
Pontryagin's stochastic maximum principle and Bellman's dynamic programming principle.
In the former approach, a necessary  condition of optimality can be obtained
under certain regularity conditions of the system. On the other hand, the latter approach results
in different versions of HJB equations, which can be used to characterize the optimal control. We refer readers to \cite{tang1994necessary} for the general stochastic maximum principle for the control system driven by jump-diffusion processes. For a systematic account of the two approaches, one may refer
to the monograph \cite{ZhYo} and the references therein.

This paper is concerned with Bellman's dynamic programming principle and the associated stochastic HJB equation with jumps.
We first study the corresponding Bellman's dynamic programming principle for Problem \ref{pro:2.1}.
Note that the initial time $t=0$ and the initial state $X(0)=x_0$ are fixed in the formulation of Problem \ref{pro:2.1}.
The basic idea of the dynamic programming principle is, however, to consider a family of optimal control problems with
different initial times and states, to establish the relation among these problems, and finally to solve all these problems via a stochastic HJB equation.

To make this idea precise, we fix a set of  initial data $(\tau, \xi) \in T_0\times  L^2(\Omega,{\mathscr F_\tau}, \mathbb P; \mathbb R^n)$.
For any given admissible control $u(\cdot)\in\mathscr V[\tau, T]$, we consider the following state equation:
\begin{eqnarray} \label{eq:3.1}
\left\{
\begin{aligned}
dX(s)=& \ b(s,X(s),u(s))ds+\sigma(s,X(s),u(s))d W(s)+ \int_Eg(s, e, X({s-}), u(s))\tilde\mu(d e, ds), ~~~~~ \tau\leq s\leq T,\\
X(\tau)=& \ \xi.
\end{aligned}
\right.
\end{eqnarray}
The cost functional is defined by
\begin{eqnarray}\label{eq:b12}
\mathbb J(\tau,\xi; u(\cdot))=Y^{\tau,\xi;u} (\tau),
\end{eqnarray}
where $(Y^{\tau,\xi;u},Z^{\tau,\xi;u},K^{\tau,\xi;u})$ is the solution of \eqref{eq:b5} on time interval $[\tau,T]$.
Then, corresponding to the control system (\ref{eq:3.1}) and the cost functional (\ref{eq:b12}), the optimal control problem parameterized by $(\tau,\xi)\in T_0\times  L^2(\Omega,{\mathscr F_\tau}, \mathbb P; \mathbb R^n)$ is formulated as follows:

\begin{pro}[$D_{\tau,\xi}$]\label{pro:3.1}
Find an admissible control process ${\bar u} (\cdot) \in \mathscr V[\tau, T]$ such that
\begin{eqnarray}\label{eq:3.3}
\mathbb J (\tau, \xi; {\bar u}(\cdot) ) = \essinf_{u (\cdot) \in {\mathscr V[\tau, T]}} \mathbb J (\tau,\xi; u (\cdot) ).
\end{eqnarray}
\end{pro}
We denote the above optimal control problem by Problem $(D_{\tau,\xi})$ to stress the dependence on the initial state $(\tau, \xi)$. From Lemma \ref{lem_esti_sde} in the next subsection,
for any initial data $(\tau, \xi) \in T_0\times  L^2(\Omega,{\mathscr F_\tau}, \mathbb P; \mathbb R^n)$ and admissible control $u(\cdot)\in \mathscr V[\tau, T],$
the  state equation \eqref{eq:3.1} has
a unique strong solution $X(\cdot)\equiv
X^{t,\xi; u}(\cdot)\in {S}^2_{\mathscr F} ( \tau, T; \mathbb R^n )$ and the cost functional
\eqref{eq:b12} is well-defined. Furthermore, we can define the
following conditional minimal value system
\begin{eqnarray}\label{eq:3.4}
\mathbb V(\tau, \xi) \triangleq \mathop{\textrm{ess.}\inf}_{u (\cdot) \in
\mathscr V[\tau, T]} \mathbb J ( \tau,\xi; u(\cdot) ) .
\end{eqnarray}
Clearly, for any $(\tau, \xi) \in T_0\times  L^2(\Omega,{\mathscr F_\tau}, \mathbb P; \mathbb R^n)$,   $\mathbb V(\tau, \xi)$
is an $\mathscr F_\tau$-measurable
random variable.

\subsection{Preliminary Results}

In this subsection, we provide some preliminary results for the controlled SDE, which are needed in the following sections. The proof of the first two lemmas can be found in \cite{kunita2004stochastic}.

\begin{lem}\label{lem_esti_sde}
Let Assumptions \ref{ass:2.1}-\ref{ass:2.2} hold. Given $\tau \in T_0 $ and $p \ge 2$, SDE \eqref{eq:3.1} admits a unique solution for any $u(\cdot) \in\mathscr V[\tau, T]$ and $\xi\in L^p(\Omega,{\mathscr F_\tau}, \mathbb P; \mathbb R^n)$. Moreover, there exists a positive constant $C_p$ such that, for any $\tau$, $u(\cdot),\bar u(\cdot)\in\mathscr V[\tau, T]$, and $\xi,\bar \xi \in L^p(\Omega,{\mathscr F_\tau}, \mathbb P; \mathbb R^n)$, it holds that
	\begin{equation}\label{eq:3.15}
	\mathbb E\left[\sup_{\tau\le s \le T}|X^{t,\xi;u}(s)|^p \bigg |\mathcal F_\tau \right]\le C_p(1+|\xi|^p)
	\end{equation}
	and
	\begin{eqnarray} \label{eq:3.15}
	\mathbb E\bigg[ \sup_{\tau\leq s\leq T}
	|X^{t, \xi; u}(s)-X^{t, \bar \xi; \bar  u}(s)|^p\bigg|\mathscr F_{\tau}\bigg]\leq C_p
	\bigg(|\xi-\bar\xi|^p+\mathbb E\bigg[\int_\tau^T|u(s)-\bar u(s)|^pds\bigg|
	\mathscr F_\tau\bigg]\bigg).
	\end{eqnarray}
	Moreover, the solution $X(\cdot)$ satisfies the flow property, i.e., for any $ t\le\tau \le \gamma$ and $u \in \mathscr V[\tau,T]$,
	$$
	X^{t,x;u}(\gamma)=X^{\tau,X^{t,x;u}(\tau);u}(\gamma), \quad a.s..
	$$
\end{lem}

\begin{lem}\label{lem-gradient}
Let Assumptions \ref{ass:2.1}-\ref{ass:c4} hold. Then, the stochastic flow $X^{t,x;u}$ is an onto homeomorphism for any $t$ a.s.. Moreover, the gradient $\partial X^{t,x;u}$ of the stochastic flow is the solution of the following SDE:
\begin{align}
d \, \partial X^{t,x;u}(s)=&\ D_x b(s,X^{t,x;u}(s),u(s))\partial X^{t,x;u}(s)ds+D_x \sigma(s,X^{t,x;u}(s),u(s))\partial X^{t,x;u}(s)dW(s)\nonumber\\
&+ \int_E D_x g(s,e,X^{t,x;u}(s-),u(s))\partial X^{t,x;u}(s-)\tilde \mu(de,ds), \quad \partial X^{t,x;u}(t)=I,
\end{align}
where $D_x$ denotes the gradient with respect to $x$.

From the a priori estimates for linear SDEs, we immediately have that, for any $p\ge 2$, there exists a constant $C_p$ such that
$$
\mathbb E\left[ |\partial X^{t,x;u}(s)|^p  \right] \le C_p.
$$
\end{lem}
We also have the following lemma concerning the solution to the BSDE \eqref{eq:b5}.
\begin{lem}\label{lem_bsde}
Let Assumptions \ref{ass:2.1}-\ref{ass:2.2} hold. The solution $(Y^{\tau,\xi,u},Z^{\tau,\xi,u},K^{\tau,\xi,u})$ of BSDE \eqref{eq:b5} satisfies
$$
|Y (\tau)|^2 \le C\mathbb E\left[ |h(X (T))|^2+\int_\tau^T|f(s,X (s),u(s),0,0,0)|^2ds  \bigg|\mathcal F_\tau\right].
$$
Hence, we have that
$$
|Y(\tau)|\le C(1+|\xi|^2).
$$
\end{lem}
\begin{proof}
The first estimate is a standard result for the a priori estimates of BSDE with jumps, see \cite{barles1997backward} and \cite{tang1994necessary}. The second one can be obtained by using Lemma \ref{lem_esti_sde}.
\end{proof}
With the help of above lemmas, we immediately have the following result.
\begin{lem}\label{lem:3.6}
	Let Assumptions \ref{ass:2.1}-\ref{ass:2.2} be satisfied. Then,
	for any given $\tau \in T_0$, $\xi, \bar \xi \in L^2(\Omega,{\mathscr F_\tau}, \mathbb P; \mathbb R^n)$,
	and $u(\cdot), \bar u(\cdot)\in \mathscr V[\tau, T]$, we have
	\begin{eqnarray} \label{eq:3.15}
	\mathbb J(\tau, \xi; u)\leq C(1+|\xi|^2),
	\end{eqnarray}
	and
	\begin{eqnarray}\label{eq:J_convergence}
	\begin{split}
	|\mathbb J(\tau,\xi; u)-\mathbb J(\tau,\bar\xi; \bar u)| \le C
	\bigg((1+|\xi|+|\bar \xi|)|\xi-\bar\xi|+\mathbb E\bigg[\int_\tau^T(1+|u(s)|+|\bar u(s)|)|u(s)-\bar u(s)|ds\bigg|
	\mathscr F_\tau\bigg]\bigg).
	\end{split}
	\end{eqnarray}
Therefore,
	\begin{eqnarray} \label{eq:3.15}
	|\mathbb V(\tau, \xi)|\leq C(1+|\xi|^2),
	\end{eqnarray}
	and
	\begin{eqnarray}\label{eq:3.16}
	|\mathbb V(\tau,\xi)-\mathbb V(\tau,\bar \xi)|\leq C(1+|\xi|+|\bar \xi|)|\xi-\bar \xi| .
	\end{eqnarray}
\end{lem}

\subsection{Dynamic Programming Principle}

In this subsection, we prove that the
conditional minimal value system
$\{\mathbb V(t,\xi)| t\in \cal T\}$
satisfies dynamic programming principle. For that purpose, we first introduce the concept of the so-called backward semigroup, which is first introduced by Peng \cite{Peng1997}.  Given the initial data $(\tau,\xi)$ with $\tau \in T$ and $\xi \in L^2(\Omega,{\mathscr F_\tau}, \mathbb P;\mathbb R^n)$, a stopping time $\gamma \in T_{\tau}$, an admissible control process $u(\cdot)\in \mathscr  V[\tau,\gamma]$, and a real-valued random variable $\eta \in L^2(\Omega,{\mathscr F_\gamma}, \mathbb P;\mathbb R)$, we define
$$
G^{\tau,\xi;u(\cdot)}_{s,\gamma}[\eta]:=\tilde Y (s), \quad s\in \llbracket \tau,\gamma\rrbracket,
$$
where $(\tilde X,\tilde Y,\tilde Z,\tilde K)$ is the solution of the following forward-backward system:
\begin{equation*}
\left\{
\begin{split}
dX(s) =& \ b(s,X(s),u(s))ds+\sigma(s,X(s),u(s))d W(s)+ \int_Eg(s, e, X({s-}), u(s))\tilde\mu(d e, ds), \\
dY(s)=&-f(s,X (s),u(s),Y(s),Z(s),\int_E K(s,e)l(s,e)\nu(de))ds+Z(s)dW(s)+\int_E K(s,e)\tilde\mu (ds,de),\quad \tau\leq s\leq \gamma,\\
X(\tau) =& \ \xi , \quad Y(\gamma)=\eta .
\end{split}
\right.
\end{equation*}
The main result is given in the following theorem.

\begin{thm}\label{thm_dpp_1}
Under Assumptions \ref{ass:2.1}-\ref{ass:2.2}, the
conditional minimal value system $\mathbb V(\tau,\xi)$ obeys the following
dynamic programming principle: for any $\tau \in T_0$, $\gamma \in T_{\tau}$, and $\xi \in L^2(\Omega,{\mathscr F_\tau}, \mathbb P,\mathbb R^n)$, it holds
\begin{eqnarray}\label{eq:3.4}
\mathbb V(\tau,\xi)= \essinf_{u(\cdot)\in\mathscr V[\tau, \gamma]}  G^{\tau,\xi;u(\cdot)}_{s,\gamma}[\mathbb V(\gamma,X^{\tau,\xi;u(\cdot)}(\gamma))].
\end{eqnarray}
\end{thm}

Before proving Theorem \ref{thm_dpp_1},  we present the following lemmas.

\begin{lem}\label{lem_convergence}
Let Assumptions \ref{ass:2.1}-\ref{ass:2.2} be satisfied.
Then for any initial data  $(\tau, \xi) \in T_0 \times L^2(\Omega,{\mathscr F_\tau}, \mathbb P;\mathbb R^n)$,
the set $\{ \mathbb J(\tau, \xi; u(\cdot)) | u(\cdot)\in \mathscr V[\tau, T]\}$ is closed under pairwise minimization.
Consequently, there exists a sequence of admissible controls $\{u_k(\cdot)\}_{k=1}^{\infty}$  such that
$\{\mathbb  J(\tau,\xi; u_k(\cdot))\}_{k=1}^{\infty}$ is non-increasing and
\begin{eqnarray}\label{eq:3.5}
\lim_{k \rightarrow \infty} \mathbb  J(\tau,\xi; u_k(\cdot))(\omega)=\mathbb V(\tau, \xi)(\omega), \quad a.e..
\end{eqnarray}
Moreover, for any sub-algebra $\mathscr G$ of $\mathscr F_t,$ the  $\mathscr G$-conditional expectation
is interchangeable with the essential infimum:
\begin{eqnarray} \label{eq:3.6}
\mathbb E[ \mathbb V (t, \xi)|\mathscr G]=
\essinf_{u(\cdot)\in\mathscr V[t, T]}  \mathbb E[ \mathbb  J(t,\xi; u(\cdot)) |\mathscr G], \quad a.e..
\end{eqnarray}
\end{lem}

\begin{proof}
 Given $u_1(\cdot), u_2(\cdot)\in \mathscr V[\tau, T]$, letting $A\triangleq\{\omega |\mathbb J(\tau, \xi; u_1(\cdot)) \leq \mathbb J(t, \xi; u_2(\cdot))\}$, we have $A\in \mathscr F_t$.
Define $v(\cdot)\triangleq u_1(\cdot)\chi_A+u_2(\cdot)\chi_{A^C}$, which is an admissible control in $\mathscr V[t, T]$.
From the uniqueness for the solution of BSDE \eqref{eq:b5}, it is easy to check that
\begin{eqnarray}
     \begin{split}
      \mathbb  J(\tau, \xi; v(\cdot))&= \mathbb J(\tau, \xi; u_1(\cdot)\chi_A+u_2(\cdot)\chi_{A^C})
       \\&=
       \mathbb J(\tau, \xi; u_1(\cdot))\chi_A+\mathbb J(\tau, \xi; u_2(\cdot))\chi_{A^C}
       \\&= \mathbb J(\tau, \xi; u_1(\cdot))\wedge \mathbb J(\tau, \xi; u_2(\cdot)).
     \end{split}
\end{eqnarray}
 Thus, the set $\{ \mathbb J(\tau, \xi; u(\cdot)) | u(\cdot)\in \mathscr V[\tau, T]\}$ is closed under pairwise minimization. Moreover, one can also get that
 $$
 \mathbb J(\tau,\xi;u(\cdot)) \ge -C(1+|\xi|^2) .
 $$
Then, according to  Remark \ref{rmk_essinf}, \eqref{eq:3.5} and \eqref{eq:3.6} follow directly from
Lemma \ref{lem:2.1}. The proof is completed.
\end{proof}

Next, we prove that one can choose an admissible control that is of at most $\varepsilon$ difference to the optimal value.
\begin{lem}\label{lem_approximate_optimal}
For any initial data  $(\tau,\xi) \in T_0\times L^2(\Omega,\mathcal F_{\tau},\mathbb P;\mathbb R^n)$ and  $\varepsilon>0$, there exists an admissible control $u^{\varepsilon} \in \mathcal V[\tau, T]$ such that
$$
\mathbb J(\tau,\xi;u^{\varepsilon}(\cdot)) \le \mathbb V(\tau,\xi)+\varepsilon,\text{ a.s..}
$$
\end{lem}
\begin{proof}
From Lemma \ref{lem_convergence}, there exists a sequence of admissible controls $\{u_k(\cdot)\}_{k=1}^{\infty}$ such that $\{\mathbb J(\tau,\xi;u_k(\cdot))\}_{k=1}^{\infty}$ is non-increasing and convergences to $\mathbb V(\tau,\xi)$ almost surely. Define the following sets as
$$
A_k:=\{ \mathbb J(\tau,\xi;u_k(\cdot))\le \mathbb V(\tau,\xi)+\varepsilon\}.
$$
Then, we have that $A_k \subset A_{k+1}$ and $\cup_{k=1}^{\infty} A_k=\Omega$. We construct the admissible control $u^{\varepsilon}$ as
$$
u^{\varepsilon}(\cdot)=\sum_{k=1}^{\infty} u_k(\cdot) 1_{ A_k \setminus A_{k-1}},
$$
with $A_0=\emptyset$. From the uniqueness for the solution of BSDE \eqref{eq:b5}, we see that
$$
\mathbb J(\tau,\xi;u^{\varepsilon}(\cdot))1_{ A_k \setminus A_{k-1}}=\mathbb J(\tau,\xi;u_k(\cdot))1_{ A_k \setminus A_{k-1}}\le \left( \mathbb V(\tau,\xi)+\varepsilon\right)1_{ A_k \setminus A_{k-1}},
$$
for any $k = 1,2,\cdots$, which leads to the desired result.
\end{proof}

\begin{proof}[Proof of Theorem \ref{thm_dpp_1}]
First, from the uniqueness for the solution of the forward-backward system, we have that for any initial data $(\tau,\xi)$, admissible control $u(\cdot)\in \mathscr V[\tau,T]$, and $\gamma \in T_{\tau}$, the following
relation holds
$$
G^{\tau,\xi;u}_{s,T}[h(X^{\tau,\xi;u}(T))]=G^{\tau,\xi;u}_{s,\gamma}[Y^{\gamma,X^{\tau,\xi;u}(\gamma);u}(\gamma)],
\quad \tau \le s \le \gamma.
$$
Hence,
$$
\mathbb V(\tau,\xi)=\essinf G^{\tau,\xi;u}_{s,T}[h(X^{\tau,\xi;u}(T))]=\essinf G^{\tau,\xi;u}_{s,\gamma}[Y^{\gamma,X^{\tau,\xi;u}(\gamma);u}(\gamma)]\ge \essinf G^{\tau,\xi;u}_{s,\gamma}[\mathbb V(\gamma,X^{\tau,\xi;u}(\gamma))].
$$
From Lemma \ref{lem_approximate_optimal}, it holds that, for any $\varepsilon>0$ and $u(\cdot) \in \mathscr V[\tau,T]$, there exists an admissible control $\bar u(\cdot) \in \mathcal[\gamma, T]$ such that
$$
\mathbb V(\gamma,X^{\tau,\xi;u}(\gamma))\ge Y^{\gamma,X^{\tau,\xi;u}(\gamma);\bar u(\cdot)}(\gamma)-\varepsilon,\text{ a.s.}.
$$
Combining the two controls $u (\cdot)$ and ${\bar u} (\cdot)$ as
$$
\tilde u(s):=\left\{
\begin{split}
u(s), \quad \tau \le s\le \gamma,\\
\bar u(s), \quad \gamma \le s \le T,
\end{split}
\right.
$$
we have
\begin{equation*}
\begin{split}
\mathbb V(\tau,\xi) \le& \ G^{\tau,\xi;\tilde u}_{\tau,T}[h(X^{\tau,\xi;\tilde u}(T))]= G^{\tau,\xi;\tilde u}_{\tau,\gamma}[Y^{\gamma,X^{\tau,\xi;\tilde u}(\gamma);\tilde u}(\gamma)]\\
\le& \ G^{\tau,\xi;\tilde u}_{\tau,\gamma}[\mathbb V(\gamma,X^{\tau,\xi;\tilde u}(\gamma))+\varepsilon]\le  G^{\tau,\xi;\tilde u}_{\tau,\gamma}[\mathbb V(\gamma,X^{\tau,\xi;\tilde u}(\gamma))]+C\varepsilon,
\end{split}
\end{equation*}
where the last inequality is due to the estimate for BSDE with jumps. From the arbitrariness of $\varepsilon$, we get the desired result.
\end{proof}

We see that $\mathbb V(\tau,\cdot)$ is a mapping from $L^2(\Omega,{\mathscr F_
\tau}, \mathbb P; \mathbb R)$ to itself. One can get a random function from this mapping by restricting $\mathbb V$ to the deterministic random variables as initial state values, i.e.,
\begin{eqnarray}\label{eq:3.13}
   V(t,x) \triangleq \mathbb V(t,x), \quad (t,x)\in [0, T]\times \mathbb R^n.
\end{eqnarray}
This random function is called the value function for the optimal control problem.  We shall see that $V (t,x)$ also satisfies dynamic programming principle.  Similarly, for any control $u(\cdot)$, we define
$$
J(t,x,u(\cdot)) \triangleq \mathbb J(t,x,\cdot), \quad (t,x)\in [0, T]\times \mathbb R^n.
$$
The following theorem is an analogous result of Theorem \ref{thm_dpp_1}. 
\begin{thm} \label{thm:3.5}
   Let Assumptions  \ref{ass:2.1}-\ref{ass:2.2}
   be satisfied. Then the value function
   $V(t, x)$ obeys  the following
  dynamic programming principle: for any $0\leq t\leq t+\delta\leq T$,
\begin{eqnarray}\label{eq:3.14}
       V(t,\xi)= \essinf_{u(\cdot)\in\mathscr V[t, T]}  G^{t,\xi;u}_{t,t+\delta}[ V(t+\delta,X^{t,\xi;u}(t+\delta))].
   \end{eqnarray}
\end{thm}
Now we present some elementary property  of
the cost functional and  the
value function. Then, Theorem \ref{thm:3.5} is an immediate result of these lemmas.
\begin{lem}\label{lem:3.9}
   Let Assumptions \ref{ass:2.1}-\ref{ass:2.2} be satisfied.
   Then, for any given $(t, \xi) \in [0,T]\times  L^2(\Omega,{\mathscr F_t}, \mathbb P; \mathbb R^n)$ and
$u(\cdot)\in \mathscr V[t, T]$,
we have
\begin{eqnarray}
  J(t,\xi; u(\cdot))=\mathbb J(t,\xi; u(\cdot)).
\end{eqnarray}
\end{lem}

\begin{proof}
We first consider   the case where $\xi$
is a simple random variable of the form
\begin{eqnarray}
  \xi=\sum_{i=1}^Nx_i\chi_{A_i},
\end{eqnarray}
where $\{A_i\}_{i=1}^N$ is a finite
 partition of $(\Omega, \mathscr F_t)$ and $x_i\in \mathbb R^n,$
 for $1\leq i\leq N.$
In view of  $$\sum^N_{i=1}\Phi(x_i)\chi_{A_i}
=\Phi(\sum^N_{i=1}x_i\chi_{A_i}),$$  we derive
\begin{align}
\sum_{i=1}^N \textbf{1}_{A_i}X^{t,x_i;u}(s)= & \ \xi+
\int_{0}^t b(s,\sum_{i=1}^N\textbf{1}_{A_i}X^{t,x_i;u}(s),u(s))ds
+\int_0^t\sigma(s,\sum_{i=1}^N \textbf{1}_{A_i}X^{t,x_i;u}(s),u(s))d W(s)\nonumber\\
&+ \int_0^t\int_Eg(s, e, \sum_{i=1}^N \textbf{1}_{A_i}X^{t,x_i;u}(s), u(s))\tilde\mu(d e, ds).
\end{align}
Thus, the uniqueness of
the solution to the above SDE leads to
$$ X^{t,\xi;u}(s)=\sum_{i=1}^N \textbf{1}_{A_i}X^{t,x_i;u}(s), \quad s\in [t, T].$$
Therefore, we have
  \begin{eqnarray}
    \begin{split}
      \mathbb J(t,\xi; u(\cdot))=&
      \mathbb E\displaystyle\bigg[\displaystyle\int_t^T
f(s,X^{t,\xi;u}(s),u(s))ds+h(X^{t,\xi;u}(T))\bigg|\mathscr F_t \bigg]
\\=& \mathbb E\displaystyle\bigg[\displaystyle\int_t^T
f(s,\sum_{i=1}^N \textbf{1}_{A_i}X^{t,x_i;u}(s),u(s))ds+h(\sum_{i=1}^N \textbf{1}_{A_i}X^{t,x_i;u}(T))\bigg|\mathscr F_t \bigg]
\\=&  \sum_{i=1}^N\textbf{1}_{A_i}\mathbb E\displaystyle\bigg[\displaystyle\int_t^T
f(s, X^{t,x_i;u}(s),u(s))ds+h(X^{t,x_i;u}(T))\bigg|\mathscr F_t \bigg]
\\=&  \sum_{i=1}^N\textbf{1}_{A_i} J(t,x_i; u) = J(t,\sum_{i=1}^N\textbf{1}_{A_i}x_i; u)=J(t,\xi; u).
    \end{split}
  \end{eqnarray}
For a general $\xi\in L^2(\Omega,{\mathscr F_t}, \mathbb P; \mathbb R^n)$, we can choose a sequence of
simple random variables $\{\xi_i\}$ such that
$$\lim_{i\rightarrow \infty}\xi_i=\xi \quad \mbox{in} \quad L^2(\Omega,{\mathscr F_t},\mathbb P; \mathbb R^n).$$
The final desired result follows from Lemma \ref{lem:3.6}.
\end{proof}

In the same vein, we obtain the following relation between $V(\cdot, \cdot)$ and $\mathbb V(\cdot, \cdot)$.

\begin{lem}\label{lem:3.10}
Under Assumptions \ref{ass:2.1}-\ref{ass:2.2}, for any given $(t, \xi) \in [0,T]\times L^2(\Omega,{\mathscr F_t}, \mathbb P; \mathbb R^n)$,
we have
\begin{eqnarray}\label{eq:3.17}
V(t, \xi)=\mathbb V(t, \xi) .
\end{eqnarray}
\end{lem}
\begin{rmk}
We see that Theorem \ref{thm:3.5} is a `weaker' version of Theorem \ref{thm_dpp_1}. The extension of Theorem \ref{thm:3.5} to any stopping times instead of $t$ and $t+\delta$ seems to be non-trivial. In many cases, it requires the continuity of the value function with respect to both $t$ and $x$. In \cite{li2009stochastic}, Li and Peng considered similar optimal control problems but with deterministic coefficients. They proved that the value function, which is a deterministic function, is $1/2$-H\"older continuous with respect to $t$. However, in our framework, it is not the case. In fact, from our later result, the value function is the solution of BSPDE and, thus, is only right continuous with left limits. In \cite{tang2015dynamic}, Tang proved that the result in Lemma  \ref{lem:3.10} holds true also for random time $\tau$. This will yield the stronger version of dynamic programming principle. But, their results rely on the linear-quadratic structure of the control problem and the aggregation of a $\mathcal T$-supermatingale family. See \cite{tang2015dynamic} and  \cite{zhang2020backward} for details.
\end{rmk}
\section{Stochastic HJB equation with jumps}

In this section, we introduce a stochastic HJB equation driven by the Brownian motion and the Poisson
random measure that is associated
with our optimal control problem \eqref{pro:2.1}. The equation is derived from the dynamic programming
principle under sufficient smoothness assumptions on the value function.
Let $H: {\cal T}\times \Omega\times \mathbb R^n \times \mathbb  R^k \times \mathbb R^{n\times n}\times \mathbb R$ be defined by
\begin{eqnarray}\label{eq:4.1}
H(t,x,u, p, q, A,k) := f(t,x,u,\sigma^T p+q,k) +\big(p,b(t,x,u)\big)+ \big ( q,\sigma (t, x, u) \big ) + \frac{1}{2} \mbox{Tr}\big[A \sigma\sigma^\top(t,x,u)\big],
\end{eqnarray}
where $\mbox{Tr}[\cdot]$ denotes the trace of a square matrix.

Now we introduce a fully nonlinear backward stochastic partial differential-integral
equation driven by the Brownian motion $B$ and the Poisson
random measure $\tilde \mu$. The differential form and the integral form of this equation is, respectively, given by
\begin{eqnarray}\label{eq:16}
\left\{
\begin{aligned} -d V(t,x) =&\displaystyle\inf_{u\in U}\bigg\{H(t,x,u, DV(t,x),D \Phi(t,x), D^2 V(t,x),\int_E \left(\mathcal I V(t,e,x,u)+\Psi(t,e,x+g(t,e,x,u))\right)l(t,e)\nu(de)) \\
&+\displaystyle\int_{E}\big[\mathcal I V(t,e,x,u)-\displaystyle
(g(t, e,x,u), D V(t,x))\big]\nu(d e)+\int_{E}\big[\mathcal I \Psi(t,e,x,u)\big]\nu(d e)\bigg\}dt\\
&-\Phi(t,x)dW(t)-\displaystyle\int_{E} \Psi (t, e,x)\tilde\mu(d e,dt),\\
V(T,x)=& \ h(x),
\end{aligned}
\right.
\end{eqnarray}
and
\begin{align}\label{eq:4.3}
&V(t,x)\\
=&\ h(x)+\displaystyle \int_t^T\inf_{u\in
U}\displaystyle\bigg\{H(s,x,u, DV(s,x),D \Phi(s,x) (s), D^2 V(s,x),\int_E \left(\mathcal I V(t,e,x,u)+\Psi(t,e,x+g(t,e,x,u))\right)l(t,e)\nu(de)) \nonumber \\
&+\int_{E}\big[\mathcal I V(s,e,x,u)-\displaystyle
(g(s, e,x,u), D V(s,x))\big]\nu(d e) \nonumber +\int_{E}\big[\mathcal I \Psi(s,e,x,u)\big]\nu(d e)\bigg\}ds \nonumber \\
&-\int_t^T\Phi(s,x)dW(s)
-\int_t^T\displaystyle\int_{E} \Psi(s, e,x)\tilde\mu(d e,ds). \nonumber
\end{align}
Here, we define the non-local operator $\mathcal I$ as
$$
\mathcal I \varphi(t,e,x,u)=\varphi(t,x+g(t,e,x,u))-\varphi(t,x).
$$
The above equation \eqref{eq:16} or \eqref{eq:4.3} is the stochastic HJB equation
with jumps associated with Problem \eqref{pro:2.1}, whose solution consists of a triplet of random fields $(V, \Phi, \Psi)$.

\begin{rmk}
When all the mappings involved in the state equation \eqref{eq:b3} and the cost functional \eqref{eq:b5}
are deterministic,  the  value function $V(t,x)$, i.e., the first component of the triplet of the random fields,
becomes a deterministic function with respect to $(t,x)$, and the corresponding stochastic HJB
degenerates to a deterministic nonlinear second-order partial differential equation,
i.e., the HJB equation in the usual sense:
\begin{eqnarray}\label{eq:b16}
\left\{
\begin{aligned} -d V(t,x) =&\displaystyle\inf_{u\in U}\bigg\{H(t,x,u, DV(t,x),0, D^2 V(t,x),\int_E \mathcal I V(t,e,x,u)l(t,e)\nu(de)) \\
&+\displaystyle\int_{E}\big[\mathcal I V(t,e,x,u)-\displaystyle
(g(t, e,x,u), D V(t,x))\big]\nu(d e)\bigg\}dt,\\
V(T,x)=& \ h(x).
\end{aligned}
\right.
\end{eqnarray}
\end{rmk}
Now we give the definition of the predictable classical solution to the stochastic HJB equation  \eqref{eq:4.3}.
\begin{defn} \label{defn:4.1}
  A triplet of random fields $ (V, \Phi, \Psi)$
  is called a predictable classical solution to the stochastic HJB equation  \eqref{eq:4.3} if
  \begin{enumerate}
  \item[(i)]for each $x\in \mathbb R^n$, $(t, \omega)\rightarrow V(t,\omega, x)$ is a predictable c\`adl\`ag process and for
  almost all $(t,\omega)\in {\cal T}\times \Omega,$ $x\rightarrow V(t,x, \omega) $ is
  twice continuously differentiable;
  \item[(ii)] for each  $x\in \mathbb R^n$, $(t, \omega)\rightarrow \Phi(t,\omega, x)$
  is a predictable process and for almost all $(t,\omega)\in {\cal T}\times \Omega,$ $x\rightarrow \Phi(t,x, \omega) $ is continuously differentiable;
  \item[(iii)] for each  $x\in \mathbb R^n $, $(t, \omega, e)\rightarrow \Psi(t,\omega, e, x)$
  is a ${\mathscr P} \otimes {\mathscr B} (E)$ measurable random field and for
  almost all $(t,\omega, e)\in {\cal T}\times \Omega
  \times E,$ $x\rightarrow  \Psi (t,x, \omega) $ is continuous;
 \item[(iv)] the triplet of random fields $ (V, \Phi, \Psi)$ satisfies
  \eqref{eq:4.3}, for all $(t, x)\in {\cal T}\times\mathbb R^n$ a.s..
\end{enumerate}

\end{defn}

\begin{prop}\label{prop:b1}
Let Assumptions \ref{ass:2.1}-\ref{ass:c4} be satisfied.
Suppose that the value function $ V(t,x)$ of Problem \ref{pro:2.1} (see \eqref {eq:3.13}) can be written as a semimartingale of the
following form:
\begin{eqnarray}
V(t,x)=h(x)+\int_t^T\Gamma(s, x)ds-\int_t^T \Psi(s,x)dW(s) -\int_t^T \int_{{E}}\Psi( s, e,x)\tilde\mu(d e,dt),
\quad (t,x)\in {\cal T} \times  \mathbb R^n,
\end{eqnarray}
where $ (V, \Phi, \Psi)$ is a given triplet of random fields satisfying the regular conditions  (i)-(iii) in Definition \ref{defn:4.1}
and the random field in the drift term, i.e., $(t,\omega, x)\rightarrow  \Gamma(t, \omega, x)$, is a given ${\mathscr P} \otimes {\mathscr B} (\mathbb R^n)$ measurable mapping. Assume that the four random fields satisfy the following regularity conditions:
 \begin{enumerate}
 \item[(a)]
$V$, $\Phi$, $\Psi$, $\Gamma$, and
their  involved partial derivatives
with respect to $x$ are continuous in $x\in \mathbb R^n$;
\item[(b)] There exists a predictable process $L \in M^{2,2}_{\mathscr{F}}(0,T;{\mathbb R})$ such that
$$
|V(t,x)|+|\Gamma(t,x)|+|\Phi(t,x)|+\int_E |\Psi(t,x,e)|\nu(de) \le L_t(1+|x|^2),
$$
$$
|D V(t,x)|+|D\Gamma(t,x)|+|D\Phi(t,x)|+\int_E |D\Psi(t,x,e)|\nu(de) \le L_t(1+|x|),
$$
and
$$
|D^2 V(t,x)| \le L_t.
$$
\end{enumerate}
If, in addition, for each $(t,x)$, the optimal control $u^{*,t,x}$ exists,
then $(V, \Phi, \Psi)$ is a classical solution to the stochastic HJB equation \eqref{eq:4.3}.
\end{prop}

\begin{proof}
To prove
$ (V, \Phi, \Psi)$ is a
classical solution to the stochastic HJB equation \eqref{eq:4.3}, by
 Definition \ref{defn:4.1}, we only need to
 show the following equality
 holds for all $(t, x)\in [0,T]\times\mathbb R^n$ a.s.,
 \begin{eqnarray}
 \Gamma (t,x)&=&\inf_{u\in U}\bigg\{H(t,x,u, DV(t,x),D \Phi(t,x), D^2 V(t,x),\int_E \left(\mathcal I V(t,e,x,u)+\Psi(t,x+g(t,e,x,u))\right)l(t,e)\nu(de)) \nonumber \\
&&+\int_{E}\big[\mathcal IV (t,e,x,u)-(g(t, e,x,u), D V(t,x))\big]\nu(d e) \nonumber +\int_{E}\big[\mathcal I \Psi(t,e,x,u)\big]\nu(d e)\bigg\},
 \end{eqnarray}
which implies that Condition $(iv)$ in  Definition \ref{defn:4.1} holds.

Let $X^{0,x;u}$  be the state process
corresponding to the control $u(\cdot) \in \mathscr V[0, T]$ for Problem \ref{pro:3.1} with the initial
data $(t,\xi)=(0, x)$. Whenever there is no risk of confusion, we abbreviate $X^{0,x;u}$ as $X$. Applying the It\^o-Ventzell formula to the value function (see \cite{chen2015semi} for It\^o-Ventzell formula with jump processes), we obtain
\begin{eqnarray}\label{eq:4.16}
&& V(t+\delta, X(t+\delta))) -  V(t,X(t)) \nonumber \\
&&= - \int_t^{t+\delta}\Gamma(s, X(s-))ds
+\int_t^{t+\delta}\Phi(s, X(s-))dW(s) +\int_t^{t+\delta}\int_{E}\Psi(s, X(s-)
+g(s,e, X(s-), u(s)))\tilde\mu(ds,d e) \nonumber \\
&& \quad +\int_t^{t+\delta}{\cal L}^u V(s, X(s-))ds+\int_t^{t+\delta}
\Big(D V(s,X(s-)),\sigma(s,x(s),u(s))dW(s)\Big) +\int_t^{t+\delta}\int_{E} \mathcal I V(s,e,X(s-),u(s))\tilde\mu(de,d s) \nonumber\\
&& \quad + \int_t^{t+\delta} \int_{E}\big[\mathcal I V(s,e,X(s-),u(s))-(DV(s,X(s-)),\sigma(s,e,X(s-),u(s)))\big]\nu (de)ds \nonumber\\
&&\quad+ \int_t^{t+\delta}\Big(D \Phi (s,X(s-)),\sigma(s,X(s),u(s))\Big)ds+ \int_t^{t+\delta}\int_{E}\mathcal I \Psi(s,e,X(s-),u(s))\nu(d e)ds,
\end{eqnarray}
with
$$
\mathcal L^{u}V=\big(DV,b(t,x,u)\big)+ \frac{1}{2} \mbox{Tr}\big[D^2V \sigma\sigma^\top(t,x,u)\big] .
$$
On the other hand,  consider the following BSDE:
\begin{align*}
&dY(s)=-f(s,X(s),u(s),Y(s),Z(s),\int_E K(s,e)l(s,e)\nu(de))ds+Z(s)dW(s)+\int_E K(s,e)\tilde\mu (ds,de), \\ &Y(t+\delta)=V(t+\delta,X(t+\delta)).
\end{align*}
From the dynamic programming principle \eqref{eq:3.14}, it holds that $V(t,X(t)) \le Y(t)$. Define
\begin{align}
F (t,x,u)=&-\Gamma(t,x)+H(t,x,u, DV(t,x),D \Phi(t,x), D^2 V(t,x),\int_E \left(\mathcal I V(t,e,x,u)+\Psi(t,x+g(t,e,x,u))\right)l(t,e)\nu(de)) \nonumber \\
&+\int_{E}\big[\mathcal IV (t,e,x,u)-(g(t, e,x,u), D V(t,x))\big]\nu(d e) \nonumber +\int_{E}\big[\mathcal I \Psi(t,e,x,u)\big]\nu(d e).
\end{align}
Let
$$Z'(s):=\Phi(s,X(s-))+\sigma(s,X(s),u(s))^\top DV(s,X(s)),$$
and
$$
K'(s,e)= \mathcal I V(s,e,X(s-),u(s))+\Psi(s,e,X(s-)+g(s,e,X(s-),u(s))).
$$
Then, we can see that $V(s,X(s))$ satisfies the following BSDE
\begin{equation*}
\begin{split}
dV(s,X(s))=&-\left(f(s,X (s),u(s),V(s,X(s)),Z'(s),\int_E K'(s,e)l(s,e)\nu(de))-F(s,X(s),u(s))\right)ds\\
&+Z'(s)dW(s)+\int_E K'(s,e)\tilde\mu (ds,de).
\end{split}
\end{equation*}
Following the argument of the comparison principle for BSDEs, we derive that
$$
0 \le Y(t)-V(t,X(t))=\mathbb E\left[ \int_t^{t+\delta}\xi_s F(s,X(s),u(s))ds \bigg |\mathcal F_t\right],
$$
where $\xi_s $ is the solution of a linear SDE
$$
d\xi_s=\alpha_s\xi_sds+\beta_s\xi_sdW(s)+\int_E \gamma(s,e)\xi_{s-}\tilde \mu(de,ds), \quad \xi_t=1,
$$
with the coefficients $\alpha$, $\beta$, and $\gamma$ being bounded processes. Their bounds are determined by the Lipschitz constants of $f$ and the bounds on $l$. Then, from the classical estimates for SDEs, we have that
$$
\mathbb E\left[ |\xi_s-1|^2 |\mathcal F_t\right]\le C|t-s|.
$$
To emphasize the dependence on $t$ and $\delta$, we denote $\xi$ as $\xi^{t,\delta}$. Then, we claim that for any $t$ and $\delta$,
\begin{eqnarray}\label{ineq_Q_expec}
\mathbb E\left[ \int_t^{t+\delta}F(s,X(s),u(s))ds\bigg |\mathcal F_t\right]\ge 0,\ \ \ \mbox{a.s.}.
\end{eqnarray}
To see this, for fixed $t$ and $\delta$, similar to the above arguments we have that for any $n$ and $k\le n$,
\begin{equation}\label{esti_linear_SDE}
\mathbb E\left[ \int_{t+\frac{k}{n}\delta}^{t+\frac{k+1}{n}\delta} \xi^{t+\frac{k}{n}\delta,\frac{\delta}{n}}_sF(s,X(s),u(s))ds\bigg |\mathcal F_t\right]\ge 0.
\end{equation}

Then, from \eqref{esti_linear_SDE}, we have
\begin{equation*}
\begin{split}
&\mathbb E\left[  \int_{t+\frac{k}{n}\delta}^{t+\frac{k+1}{n}\delta}F(s,X(s),u(s))ds\bigg|\mathcal F_t\right]\\
=&\ \mathbb E\left[ \int_{t+\frac{k}{n}\delta}^{t+\frac{k+1}{n}\delta} \xi ^{t+\frac{k}{n}\delta,\frac{\delta}{n}}_sF(s,X(s),u(s))ds \bigg|\mathcal F_t\right]+\mathbb E\left[ \int_{t+\frac{k}{n}\delta}^{t+\frac{k+1}{n}\delta}(1- \xi ^{t+\frac{k}{n}\delta,\frac{\delta}{n}}_s)F(s,X(s),u(s))ds\bigg| \mathcal F_t\right]\\
\ge &\ \mathbb E\left[ \int_{t+\frac{k}{n}\delta}^{t+\frac{k+1}{n}\delta}(1-\xi ^{t+\frac{k}{n}\delta,\frac{\delta}{n}}_s)F(s,X(s),u(s))ds\bigg|\mathcal F_t\right]\\
\ge & \left( \mathbb E\left[  \int_{t+\frac{k}{n}\delta}^{t+\frac{k+1}{n}\delta}(1-\xi ^{t+\frac{k}{n}\delta,\frac{\delta}{n}}_s)^2ds \bigg |\mathcal F_t\right] \right)^{1/2}\left( \mathbb E\left[   \int_{t+\frac{k}{n}\delta}^{t+\frac{k+1}{n}\delta}| F(s,X(s),u(s))|^2ds\bigg|\mathcal F_t  \right] \right)^{1/2} .
\end{split}
\end{equation*}
Summing over $k$, we have
\begin{equation}\label{DMZ2}
\begin{split}
&\mathbb E\left[  \int_{t}^{t+\delta}F(s,X(s),u(s))ds \bigg|\mathcal F_t\right]\\
\ge& \sum_{k=0}^{n-1} \left( \mathbb E\left[  \int_{t+\frac{k}{n}\delta}^{t+\frac{k+1}{n}\delta}(1-\xi ^{t+\frac{k}{n}\delta,\frac{\delta}{n}}_s)^2ds \bigg |\mathcal F_t\right] \right)^{1/2}\left( \mathbb E\left[   \int_{t+\frac{k}{n}\delta}^{t+\frac{k+1}{n}\delta}| F(s,X(s),u(s))|^2ds\bigg|\mathcal F_t  \right] \right)^{1/2}\\
\ge &\left( \sum_{k=0}^{n-1} \mathbb E\left[  \int_{t+\frac{k}{n}\delta}^{t+\frac{k+1}{n}\delta}(1-\xi ^{t+\frac{k}{n}\delta,\frac{\delta}{n}}_s)^2ds \bigg |\mathcal F_t\right] \right)^{1/2} \left( \mathbb E\left[   \int_{t}^{t+\delta}| F(s,X(s),u(s))|^2ds\bigg|\mathcal F_t  \right] \right)^{1/2},
\end{split}
\end{equation}
where the last inequality is obtained due to H\"older's inequality.
For fixed $t\in[0,T]$ and any nonnegative random variable $\eta\in\mathscr{F}_t$, it follows from \eqref{ineq_Q_expec} that
\begin{eqnarray*}
	\mathbb E\left[ \int_{0}^{T}F(s,X (s), u (s))\eta I_{[t,t+\delta)}(s)ds \right]=\mathbb E\left\{\eta\mathbb E\left[\int_{t}^{t+\delta}F(s,X (s), u (s))ds \bigg|\mathscr{F}_t\right]\right\}\geq0.
\end{eqnarray*}
Consequently, for any nonnegative simple process $\phi_. \in M_{\mathscr{F}}^2(0,T;\mathbb{R})$,
\begin{eqnarray*}
	\mathbb E\left[ \int_{0}^{T}F(s,X (s), u (s))\phi_sds \right]\geq0.
\end{eqnarray*}
For any nonnegative process $\psi_.\in M_{\mathscr{F}}^2(0,T;\mathbb{R})$, there exists a sequence of nonnegative simple processes $\phi^n_.\in M_{\mathscr{F}}^2(0,T;\mathbb{R})$, $n\in\mathbb{N}$, such that
\begin{eqnarray*}
	\lim_{n\to\infty}\mathbb E\left[ \int_{0}^{T}|\phi^n_s-\psi_s|^2ds \right]=0.
\end{eqnarray*}
Hence,
\begin{align*}
	&\lim_{n\to\infty}\left|\mathbb E\left[ \int_{0}^{T}F(s,X (s), u (s))\phi^n_sds \right]-\mathbb E\left[ \int_{0}^{T}F(s,X (s), u (s))\psi_sds \right]\right|\\
	&\leq\lim_{n\to\infty}\left( \mathbb E\left[   \int_{0}^{T}|F(s,X (s), u (s))|^2ds\right] \right)^{1/2}\left( \mathbb E\left[   \int_{0}^{T}|\phi^n_s-\psi_s|^2ds\right] \right)^{1/2}=0,
\end{align*}
which implies that
\begin{eqnarray*}
	\mathbb E\left[ \int_{0}^{T}F(s,X (s), u (s))\psi_sds \right]\geq0.
\end{eqnarray*}
Noting the arbitrariness of the nonnegative process $\psi_.$, we have that
\begin{eqnarray*}
	F(s,X (s), u (s)) \ge 0, \ \ \ {\rm for\ a.e.}\ s\in[0,T],\ {\rm a.s.}.
\end{eqnarray*}
Given an admissible control $u$, let $\mathbb X^x(s)$ be the stochastic flow generated by the SDE \eqref{eq:3.1} with the initial condition $X(0)=x$.
From Lemma \ref{lem_esti_sde}, with probability $1$, for each $s$, $\mathbb  X^\cdot(s)$ is a diffeomorphism of class $C^1$. For each $x_i$, we also have that
$$
F(s,\mathbb X^{x_i}(s),u(s)) \ge 0, \ \ \ {\rm for\ a.e.}\ s\in{\cal T},\ {\rm a.s.}.
$$
Thus, it holds that for all $x_i$,
$$
F(s,\mathbb X^{x_i}(s),u(s)) \ge 0, \ \ \ {\rm for\ a.e.}\ s\in{\cal T},\ {\rm a.s.}.
$$
Since $F(s,x,u)$ and ${\mathbb  X_s^{x}}$ are continuous with respect to $x$, we obtain that for all $x$,
$$
F(s,\mathbb X^{x}(s),u(s)) \ge 0, \ \ \ {\rm for\ a.e.}\ s\in{\cal T},\ {\rm a.s.}.
$$
From the growth condition of the coefficients and the value function, we see that
$$
|F(t,\mathbb X^{x}(t),u(t))|^2 \le C(1+L^2_t)(1+|\mathbb X^x(t)|^4).
$$
Then,
\begin{align*}
\mathbb E\left[  \int_0^T |F(t,\mathbb X^{x}(t),u(t))|^2dt\right]
&\le C\mathbb E\left[  \int_0^T (1+L_t^2) (1+|\mathbb X^{x}(t)|^4)dt\right]\\
&\le C\mathbb E\left[ \sup_{0\le t\le T} (1+|\mathbb X^{x}(t)|^4) \int_0^T (1+L_t^2)dt\right]\\
&\le C\left( \mathbb E\left[\left( \sup_{0\le t\le T}(1+|\mathbb X^{x}(t)|^4)\right)^2 \right] \right)^{1/2}\left( \mathbb E\left[\left(  \int_0^T (1+L_t^2)dt\right)^2\right]  \right)^{1/2}\\
&\le C(1+|x|^4).
\end{align*}

Now, let $\varphi$ be a smooth function such that
$$
\varphi(x)=\left\{
\begin{split}
&1, &\text{for $|x|\le 1$;}\\
&0, &\text{for $|x| \ge 2$;}\\
&\in [0,1], &\text{otherwise}.
\end{split}
\right.
$$
For $s\in[0,T]$, define $\tilde{\mathbb X}^\cdot(s)$ to be the inverse function of ${\mathbb X}^\cdot(s)$ and consider a random function
$$
g(s,x)=\xi(\mathbb X^{x}(s))\varphi(\frac{x}{N})\varphi(\frac{\tilde{\mathbb X}^x(s)}{N})|\det \partial \tilde{\mathbb X}^x(s)|^{-1}p_s,
$$
where $N\in\mathbb{N}$, $p$ is an arbitrary non-negative, bounded, adapted process, and $\xi$ is a smooth non-negative function with a compact support. We first prove that
$\mathbb E[ \int_0^T \int_{\mathbb{R}^n} F (s,\mathbb X^{x}(s))g(s,x)dxds ]$
is integrable. By H\"older's inequality, it holds that
\begin{equation*}
\begin{split}
\mathbb E\left[ \int_0^T \int_{\mathbb{R}^n}|F (s,\mathbb X^{x}(s),u(s))g(s,x)|dxds  \right]\le &\left( \mathbb E\left[ \int_0^T \int_{\mathbb{R}^n}|F(s,\mathbb X^{x}(s),u(s)) |^2 \varphi(\frac{x}{N})dxds  \right]  \right)^{1/2}\\
&\times\left( \mathbb E\left[ \int_0^T \int_{\mathbb{R}^n}|\det \partial_x \tilde{\mathbb X}^x(s)|^{-2}\xi^2(\mathbb X^{x}(s))\varphi(\frac{x}{N})\varphi^2(\frac{\tilde{\mathbb X}^x(s)}{N})p^2_s  dxds  \right] \right)^{1/2} .
\end{split}
\end{equation*}
For the first term on the right hand side, we have
$$
\mathbb E\left[ \int_0^T \int_{\mathbb{R}^n}|F(s,\mathbb X^{x}(s),u(s))|^2 \varphi(\frac{x}{N})dxds  \right]\le
\int_{|x|\le N+2} \mathbb E\left[  \int_0^T |F(s,\mathbb X^{x}(s),u(s))|^2ds\right]dx<\infty.
$$
Note that $\mathbb X^{\tilde{\mathbb X}^x(s)}(s)=x$. Hence
$
\partial_y\mathbb X^{y}(s)|_{y=\tilde{\mathbb X}^x(s)}\partial_x\tilde{\mathbb X}^x(s)=I
$,
and thus $
|\det\partial_x\tilde{\mathbb X}^x(s)|^{-1}=|\det\partial_y\mathbb X^{y}(s)|_{y=\tilde{\mathbb X}^x(s)}|
$.
For the second term, it holds that
\begin{align*}
\mathbb E\left[ \int_0^T \int_{\mathbb{R}^n}|\det\partial_x\tilde{\mathbb X}^x(s)|^{-2}\xi^2(\mathbb X^{x}(s))\varphi(\frac{x}{M})\varphi^2(\frac{\tilde{\mathbb X}^x(s)}{N})p^2_s  dxds  \right]
&\le  C\mathbb E\left[ \int_0^T\int_{\mathbb{R}^n}|\det\partial_y\mathbb X^{y}(s)|_{y=\tilde{\mathbb X}^x(s)}|^{2}\varphi(\frac{\tilde{\mathbb X}^x(s)}{N})dxds  \right]\\
&=C\mathbb E\left[ \int_0^T \int_{\mathbb{R}^n}|\det\partial_x\mathbb X^{x}(s)|^2\varphi(\frac{x}{N}) |\det \partial_x \mathbb X^{x}(s)|dxds  \right]\\
&\le C \int_{|x| \le N+2} \mathbb E\left[ \int_0^T |\det \partial_x\mathbb X^{x}(s)|^3ds  \right]dx <\infty .
\end{align*}
This confirms that
$
\mathbb E[ \int_0^T \int_{\mathbb{R}^n} F (s,\mathbb X^{x}(s),u(s))g(s,x)dxds ]
$
is integrable. Then we have
\begin{align*}
0 &\le \mathbb E\left[ \int_0^T \int_{\mathbb{R}^n} F (s,\mathbb X^{x}(s),u(s))g(s,x)dxds  \right]\\
&= \mathbb E\left[ \int_0^T \int_{\mathbb{R}^n} F (s,\mathbb X^{x}(s),u(s))\xi(\mathbb X^{x}(s))\varphi(\frac{x}{N})\varphi(\frac{\tilde{\mathbb X}^x(s)}{N})|\det \partial_x \tilde{\mathbb X}^x(s)|^{-1}p_s dxds  \right]\\
&= \mathbb E\left[ \int_0^T \int_{\mathbb{R}^n}F (s,x,u (s))\xi(x)\varphi(\frac{\tilde{\mathbb X}^x(s)}{N})\varphi(\frac{\tilde{\mathbb X}^{\tilde{\mathbb X}^x(s)}(s)}{N})p_s dxds  \right].
\end{align*}
Letting $N \rightarrow +\infty$, the above inequality reduces to
$$
\mathbb E\left[ \int_0^T \int_{\mathbb{R}^n} F (s,x,u(s))\xi(x) p_s dxds  \right] \ge 0.
$$
From the arbitrariness of $\xi$, $p$, and $u$, we have that
\begin{equation}\label{inf_less}
\inf_u F (s,x,u) \ge 0, \quad \text{for all $x$, $ds\times P$-a.s..}
\end{equation}

Next, we show that the equality holds. Given any $(t,x) \in {\cal T}\times\mathbb R^n$, let $u^{*,t,x}$ be the corresponding optimal control. From dynamic programming principle \eqref{eq:3.14}, we see that the equality holds in \eqref{ineq_Q_expec} when we replace the arbitrary control $u$ with the optimal control $u^{*,t,x}$. Following previous arguments, we see that
$$
F(s,X^{u^{*,t,x};x} (s);u^{*,t,x} (s))=0, \quad \text{for a.e. $s \in[t,T]$, a.s..}
$$
Denote by $F(s,x):=\inf\limits_u F(s,x,u)$. Then, we see that
$$
F(s,x,0) \ge F(s,x) \ge 0.
$$
This implies that
$$
|F(s,x)|\le |F(s,x,0)|\le CL_t(1+|x|^2),
$$
which further yields that $F(\cdot,x) \in M^{2,1}_{\mathscr{F}}(0,T;{\mathbb R})$, for any $x$. Let $\zeta(t)$ be a mollifier defined on $[0,+\infty)$, i.e.,
$$\zeta(t)=\left \{\begin{aligned}&C e^{-\frac{1}{1-t^2}}, &\text{ if $ t \le 1$}, \\ &0, &\text{ otherwise}, \end{aligned}
\right .$$
with the constant $C$ selected so that $\int_0^{\infty} \zeta(t)dt=1$ and $\zeta_n(t)=n\zeta(nt)$. Define
$$
F_n(s,x)=\int_0^{\infty} \zeta_n(\tau)F(s+\tau,x)d\tau.
$$
We have that as $n \rightarrow +\infty$,
\begin{equation}\label{ineq_delta}
\mathbb E\left[  \int_0^T F_n(s,x)ds \right] \rightarrow \mathbb E\left[  \int_0^T F(s,x)ds \right] .
\end{equation}
Note that
\begin{equation*}
\begin{split}
F_n(s,x)=&\int_0^{\infty} \zeta_n(\tau)F(s+\tau,x)d\tau\\
\le& \int_0^{\infty} \zeta_n(\tau)F(s+\tau,x,u^{*,s,x}({s+\tau}))d\tau\\
=& \int_0^{\infty} \zeta_n(\tau)\big(F(s+\tau,x,u^{*,s,x}({s+\tau}))-F(s+\tau,X^{u^{*,s,x};x}({s+\tau}),u^{*,s,x}({s+\tau}))\big)d\tau .
\end{split}
\end{equation*}
For simplicity, we abbreviate $(X^{u^{*,s,x};x},u^{*,s,x})$ as $(X^*,u^*)$. From the assumptions of the proposition, we have
\begin{align*}
&|F(s+u,x,u^{*,s,x})-F(s+u,X^{u^{*,s,x};x}({s+u}),u^{*,s,x}({s+u}))| \\
&\le CL_{s+u}(1+|x|+|X^*({s+u})|+|u^*({s+u})|)|X^*({s+u})-x|.
\end{align*}
Hence,
\begin{align*}
&\mathbb E\left| \int_0^{\infty} \zeta_n(u)(F(s+u,x,u^{*,s,x})-F(s+u,X^{u^{*,s,x};x}({s+u}),u^{*,s,x}({s+u})))du \right|\\
&\le C\left(\mathbb E\left[\int_0^{\infty} \zeta_n(u)  L_{s+u}(1+|x|+|X^*({s+u})|+|u^*({s+u})|)^2du \right]\right)^{1/2}\\
&\quad \times \left( \mathbb E\left[   \int_0^{\infty} \zeta_n(u)  L_{s+u} |X^*({s+u})-x|^2du\right] \right)^{1/2}\\
&\le C\left(\mathbb E\left[\int_0^{\infty} \zeta_n(u)  L^2_{s+u}du \right]\right)^{1/2} \left(\mathbb E\left[\int_0^{\infty} \zeta_n(u) (1+|x|+|X^*({s+u})|+|u^*({s+u})|)^4du \right]\right)^{1/4}\\
&\quad \times \left( \mathbb E\left[   \int_0^{\infty} \zeta_n(u) |X^*({s+u})-x|^4du\right] \right)^{1/4} .
\end{align*}
Then, we see that for all $s$,
$$
\mathbb E\left[   \int_0^{\infty} \zeta_n(u) |X^*({s+u})-x|^4du\right] \rightarrow 0 ,
$$
and
$$
\mathbb E\left[\int_0^{\infty} \zeta_n(u) (1+|x|+|X^*({s+u})|+|u^*({s+u})|)^4du \right]
$$
is uniformly bounded with respect to $n$. Moreover, it holds that, for almost $s$ and as $n \rightarrow \infty$,
$$
\mathbb E\left[   \int_0^{\infty} \zeta_n(u) L^2_{s+u}du\right] \rightarrow \mathbb E\left[ L_s^2\right].
$$
Hence,
$$
\liminf_{n\rightarrow\infty} \mathbb E\left[  F_n(s,x) \right]\le 0,
$$
for almost $s$.
From \eqref{ineq_delta}, we have
$$
\mathbb E\left[  \int_0^T F(s,x)ds \right] \le 0.
$$
Combining with the fact that $F (s,x) \ge 0$, we obtain that
$$
F (s,x)=0.
$$
The proof is completed.
\end{proof}

Next, we prove the verification theorem. That is, a classical solution of the backward HJB equation is the value function and characterizes the optimal control. The statement of this result is heavy, but its proof is standard and relies essentially on the It\^o-Ventzell formula.

\begin{thm}
Let Assumptions \ref{ass:2.1}-\ref{ass:c4} be satisfied. Suppose that a triplet of random fields
$(\varphi, \phi, \psi)$ is a classical solution to the stochastic HJB equation \eqref{eq:4.3}, i.e.,
\begin{align}\label{eq:5.1}
\varphi(t,x)=&\ h(x)+\displaystyle \int_t^T\displaystyle\inf_{u\in U}\bigg\{H(s,x,u, D\varphi(s,x),D \phi(s,x), D^2 \varphi(s,x),\int_E \left(\mathcal I\varphi(t,e,x,u)+\psi(t,e,x+g(t,e,x,u))\right)l(t,e)\nu(de)) \nonumber \\
&+\int_{E}\big[\mathcal I\varphi (t,e,x,u)-(g(s, e,x,u), D \varphi(s,x))\big]\nu(d e)+\int_{E}\big[\mathcal I \psi(t,e,x,u)\big]\nu(d e)\bigg\}ds \nonumber \\
&-\int_t^T\phi(s,x)dW(s)-\int_t^T\int_{E}\psi(s, e,x)\tilde\mu(d e,ds),
\end{align}
and satisfies the regularity condition (b) in Proposition \ref{prop:b1}.
Moreover, for almost all $(t, \omega, x)\in [0,T)\times \Omega\times \mathbb R^n$, the infimum in \eqref{eq:5.1} is
achieved at a ${\mathscr P}\otimes \mathscr{B}(\mathbb R^n)/\mathscr{B}(\mathbb R^k)$-measurable random field
$\bar u:(t,\omega,x)\rightarrow \bar u(t, \omega,x)$ taking values in $U$ such that for any given initial state $X(t)=x$,
the following feedback control system
\begin{eqnarray}\label{eq:5.2}
\left\{
\begin{aligned}
d X(s)=& \ b(s,X(s),\bar u(s, X(s))ds+\sigma(s,X(s),\bar u(s,X(s))d W (s) \\
& +\int_Eg(s, e, X({s-}), \bar u(s, X(s-)))\tilde\mu(d e, ds), \quad t\leq s\leq T, \\
X(t) =& \ x,
\end{aligned}
\right.
\end{eqnarray}
has a unique strong solution $\bar X(\cdot)$ and $ (\bar u(\cdot, \bar X(\cdot)); \bar X(\cdot))$
is an admissible pair. Then $\phi(t,x)=V(t,x)$ for all $(t,x)\in [0, T]\times \mathbb R^n$ a.s.,
and the feedback control $\bar u(\cdot,\bar X(\cdot))$ is an optimal control,
i.e., $V(t,x)=J(t,\bar u(\cdot,\bar X(\cdot)))$.
\end{thm}

\begin{proof}
Let $(v(\cdot),z(\cdot))$ be an arbitrary admissible control pair of Problem $(D_{t,x})$. That is, $z(\cdot)$ solves the
following stochastic differential equation:
\begin{eqnarray} \label{eq:5.3}
\left\{
\begin{aligned}
dz(s)=& \ b(s,z(s),v(s))ds+\sigma(s,z(s),v(s))dW(s)
+\int_Eg(s, e, z({s-}), v(s))\tilde\mu(d e, ds), \quad t\leq s\leq T,\\
z(t)=& \ x .
\end{aligned}
\right.
\end{eqnarray}
Define
\begin{align*}
\Delta(t,x,u):=&\ H(s,x,u, D\varphi(s,x),D \phi(s,x), D^2 \varphi(s,x),\int_E \left(\mathcal I\varphi(t,e,x,u)+\psi(t,e,x+g(t,e,x,u))\right)l(t,e)\nu(de)) \nonumber \\
&+\int_{E}\big[\mathcal I\varphi (t,e,x,u)-(g(s, e,x,u), D \varphi(s,x))\big]\nu(d e)+\int_{E}\big[\mathcal I \psi(t,e,x,u)\big]\nu(d e).
\end{align*}
By applying the It\^o-Ventzell formula to the random field $\varphi (\cdot,x)$ and the state process $z(\cdot)$ (see \eqref{eq:5.1} and \eqref{eq:5.3}), we get
\begin{align}\label{eq:5.4}
\varphi(T,z(T))=&\ \varphi(t,x)-\int_t^{T} \inf_u \Delta(s,z(s-),v(s))dt+\int_t^{T}(D\varphi(s, z(s)),b(s,z(s),v(s)))ds
+ \int_t^{T}(D\phi(s, z(s)),\sigma(s,z(s),v(s)))ds\nonumber\\
&+\frac{1}{2} \int_t^T \mbox{Tr}[D^2\varphi(s, z(s))\sigma \sigma^T(s,z(s), v(s))]ds
+\int_t^T\int_{E}\big[\varphi(s, z(s)+g(s,e,z(s),v(s)))\nonumber \\
&- \varphi(s,z(s)) -(g(s, e,z(s),v(s)), D \varphi(s,z(s)))\big]\nu(d e) \nonumber \\
&+\int_t^T\int_{E}\big[\psi(s,e,z(s))+g(s, e,z(s),v(s)))-\psi(s, e,z(s))\big]\nu(d e)ds\nonumber\\
&+\int_t^T \big[\sigma^T(s,z(s),v(s))D\varphi(s,z(s))+\phi(s,z(s))\big]dW(s)\nonumber\\
&+\int_t^T\int_E \big[\mathcal I \varphi(t,e,z(s-),v(s))+\psi(s,e,z(s-)+g(s,e,z(s),u(s))) \big]\tilde \mu(de,ds)\nonumber\\
=&\ \varphi(t,x)+\int_t^T \big[ \Delta(s,z(s-),v(s))-\inf_u \Delta(s,z(s-),v(s)) \big] ds\nonumber\\
&-\int_t^T f(s,x(s),z(s),\sigma^T(s,z(s),v(s))D\varphi(s,z(s))+\phi(s,z(s)),\nonumber\\
&\qquad \qquad \int_E \left(\mathcal I \varphi(t,e,z(s-),v(s))+\psi(s,e,z(s-)+g(s,e,z(s),u(s)))\right)l(s,e)\nu(ds))ds\nonumber\\
&+\int_t^T \big[\sigma^T(s,z(s),v(s))D\varphi(s,z(s))+\phi(s,z(s))\big]dW(s)\nonumber\\
&+\int_t^T\int_E \big[\mathcal I \varphi(t,e,z(s-),v(s))+\psi(s,e,z(s-)+g(s,e,z(s),u(s)))\big] \tilde \mu(de,ds).
\end{align}
From the comparison principle for BSDEs with jumps, the above inequality leads to
\begin{align}\label{eq:5.5}
\varphi(t, x)
&\leq G^{t,x;u(\cdot)}_{t,T}[\varphi(T,z(T))] =G^{t,x;u(\cdot)}_{t,T}[h(z(T))] =J(t,x;v(\cdot)).
\end{align}
Since $v(\cdot)$ is arbitrary, taking the infimum in \eqref{eq:5.5} gives
\begin{eqnarray}\label{eq-5.6}
\varphi(t,x)\leq V(t,x) , \quad \mbox{a.s.}.
\end{eqnarray}

Finally, again applying the It\^o-Ventzell formula to the random
field $\varphi (\cdot,x)$ (see \eqref{eq:5.1}) and
the state process $\bar X(\cdot)$
associated with the feedback control $ \bar u(\cdot, \bar X(\cdot))$ and
taking conditional expectation
with $\mathscr F_t$, we obtain the equality in
\eqref{eq:5.4},
thereby
\begin{eqnarray}
    \varphi(t, x)=J(t,x;\bar u(\cdot,\bar X(\cdot))).
\end{eqnarray}
Therefore, from \eqref{eq-5.6}
together with the definition of the value function
$V(t,x)$ (see \eqref{eq:3.3}), we have
\begin{eqnarray}\label{eq:5.6}
V(t,x)\leq J(t,x;\bar u(\cdot,\bar X(\cdot))) =\varphi(t,x)\leq V(t,x) .
\end{eqnarray}
Consequently, we conclude that $\varphi(t,x)$ coincides with the value function $ V(t,x)$
and $ (\bar u(\cdot, \bar X(\cdot)), \bar X(\cdot))$ is an optimal pair.
\end{proof}


\section{Backward stochastic evolution equation with jump}
As in the deterministic case, the classical solution of backward HJB equation does not exist in general cases. Thus, this section is devoted to the existence and uniqueness result for the stochastic HJB equation
with jumps in the sense of Sobolev spaces. To this end, we need to recast the stochastic HJB equation with jumps
as a class of backward stochastic evolution equations with jumps in Hilbert
spaces. We refer readers to \cite{PrZa92} for the general theory of stochastic evolution equations in
Hilbert spaces.


\subsection{Backward stochastic evolution equation with jumps}

We first introduce the framework of a Gelfand triple under which the backward stochastic evolution equation will be studied. The Brownian motion $B$ and the Poisson random measure $\tilde \mu$ are defined the same as in previous sections.

Let $V$ and $H$ be two separable (real) Hilbert spaces such that $V$ is densely embedded in $H$.
The space $H$ is identified with its dual space by the Riesz mapping. Then  we can take $H$ as a pivot space and  get a
Gelfand triple $V \subset H= H^*\subset V^{*},$ where  $H^*$ and
$V^{*}$ denote the dual spaces of $H$ and $V$, respectively. Denote
by $\|\cdot\|_{V},\|\cdot\|_{H}$, and $\|\cdot\|_{V^*}$ the norms of
$V,H$, and $V^*$, respectively, by $(\cdot,\cdot)_H$ the inner
product in $H$, and by $\la\cdot,\cdot\ra$ the duality product between
$V$ and $V^{*}$. Moreover, we write $\mathscr{L}(V,V^*)$ the space of bounded
linear transformations of $V$ into $V^*$.

Now we recall a version of It\^o's formula in Hilbert space which will be frequently used in this section (see \cite{gyongy1982stochastics} for the proof).

\begin{lem}\label{lem:c1}
Let $\varphi\in L^{2}(\Omega,\mathscr{F}_{0},P; H)$. Let $Y, Z$, and $\Gamma$  be three
progressively measurable stochastic processes defined on ${\cal T}\times \Omega$ with values in $V$, $H$, and $V^{*}$
such that $ Y\in M_{\mathscr{F}}^2 (0, T; V)$, $Z\in M_{\mathscr{F}}^2 (0, T; H)$, and $\Gamma \in M_{\mathscr{F}}^2 (0, T; V^*)$,
respectively. Let $R$ be a ${\mathscr P} \otimes {\mathscr B} ({E})$-measurable stochastic process defined
on ${\cal T}\times
\Omega\times E$ with values in $H$ such that $R \in M_{\mathscr F}^{\nu, 2}(0,T; H)$.
Suppose that for every $\eta \in V$ and almost every $(t,\omega)\in{\cal T}\times\Omega$, it holds that
\begin{eqnarray*}
    ( \eta,Y)_H =( \eta, \varphi)_H+
    \int_{0}^{t} \la \eta,\Gamma(s) \ra ds
    + \int_0^t( \eta, Z )_HdW(s)+\int_0^t
    \int_{E}( \eta, R(s,e) )_H\tilde \mu(de,ds) .
\end{eqnarray*}
Then, $Y$ is an $H$-valued strongly c\`adl\`ag $\mathscr F_t$-predictable process, satisfying
\begin{eqnarray*}
\mathbb E \bigg [ \sup_{0\leq t\leq T} ||Y||_H^2 \bigg ] \leq \infty ,
\end{eqnarray*}
and the following It\^{o}'s formula holds for the
squared $H$-norm of $Y$:
\begin{eqnarray}\label{eq:6.2}
     ||Y||_H^2&=&||\varphi||^2+
     2\int_0^t\langle \Gamma(s), Y(s) \rangle
     ds +2\int_0^t\langle Z(s), Y(s) \rangle
     dW(s)+ \int_0^t||Z(s)||_H^2ds \nonumber \\
     &&+\int_0^t\int_E\big[ ||R(s,e)||_H^2+
     2 (Y(s), R(s,e)) \big]
     \tilde \mu(de,ds)+\int_0^t\int_E ||R(s,e)||_H^2\nu(de)ds .
\end{eqnarray}
\end{lem}

Now we introduce a backward stochastic evolution equation with jumps (BSEEJ) in the Gelfand triple $(V, H, V^*)$ of  the following form:
\begin{eqnarray}\label{bseej}
\left\{
\begin{aligned}
d Y(t) =& \ \big [ A (t) Y (t) +B(t)Z(t)+ F ( t, Y (t), Z (t), R(t, \cdot) ) \big ] d t \\
&+ Z (t) d W (t) +  \int_E R(t,e)\tilde \mu(de,dt) ,  \quad t \in [0, T] , \\
Y(T) =& \ \xi ,
\end{aligned}
\right.
\end{eqnarray}
where the coefficients $( A, B, F, \xi )$ are given mappings such that
$A: [ 0, T ] \times \Omega \rightarrow {\mathscr L} ( V, V^* )$ is
${\mathscr P} / {\mathscr B} ( {\mathscr L} ( V, V^* ) )$-measurable; $B: [ 0, T ] \times \Omega \rightarrow {\mathscr L} ( H, V^* )$ is
${\mathscr P} / {\mathscr B} ( {\mathscr L} ( H, V^* ) )$-measurable;
$F: [0, T] \times \Omega \times V \times H \times M^{\nu,2}(E;H) \rightarrow H$
is ${\mathscr P} \otimes {\mathscr B} (V) \otimes\mathscr B(H)\otimes {\mathscr B}(M^{\nu,2}(E;H)) / {\mathscr B} (H)$-measurable; $\xi: \Omega \rightarrow H$ is ${\mathscr{F}}_T$-measurable. Furthermore, we assume that the coefficients $(A, B, F, \xi)$
satisfy the following conditions:

\begin{ass}\label{ass:6.1}
\begin{enumerate}
\item[]
\item[(i)] $F ( \cdot, 0, 0,0 ) \in M_{\mathscr F}^2 ( 0, T; H )$ and
$\xi \in L^2(\Omega,{\mathscr F}_T,P;H)$;
\item[(ii)] the operators $A$ and $B$ satisfy the super-parabolic condition, i.e., there exist constants $\alpha >0$
and $\lambda $ such that
\begin{eqnarray}\label {eq:2.9}
2\left < A (t) \phi, \phi \right > + \lambda \| \phi \|^2_H \geq \alpha \| \phi \|^2_V +||B^*\phi||^2_H ,
\quad \forall t \in [ 0, T ] , \quad \forall \phi \in V ;
\end{eqnarray}
\item[(iii)] the operators $A$ and $B$ are uniformly bounded, i.e., there exists a constant $C >0$ such that
\begin{eqnarray} \label{eq:6.5}
&\sup\limits_{( t, \omega )\in [ 0, T ] \times \Omega} \| A ( t, \omega ) \|_{{\mathscr L} ( V, V^* )}
+\sup\limits_{( t, \omega )\in [ 0, T ] \times \Omega} \| B( t, \omega ) \|_{{\mathscr L} ( H, V^* )} \leq C ;
\end{eqnarray}
\item[(iv)] $F$ is uniformly Lipschitz continuous in $( y, z,  r )$, i.e., there exists
a constant $C > 0$ such that for all $( y, z, r ), ( {\bar y}, {\bar z}, {\bar r}) \in V \times H  \times M^{\nu,2 }(E;H)$
and a.e. $( t, \omega ) \in {\cal T} \times \Omega$,
\begin{equation}\label{eq:2.18}
\| F ( t, y, z,r ) - F ( t, {\bar y}, {\bar z}, {\bar r} ) \|^2_H
\leq C ( \| y - {\bar y} \|^2_V + \| z- {\bar z} \|^2_H + \| r - {\bar r} \|^2_{M^{\nu,2}(E;H)}) .
\end{equation}
\end{enumerate}
\end{ass}

For any set of $(A,B, F,\xi)$ satisfying Assumption \ref{ass:6.1}, we call it a
generator of the BSEEJ \eqref{bseej}.

\begin{defn}\label{defn:6.1}
A $V \times H \times M^{\nu,2}(E;H)$-valued, ${\mathbb F}$-predictable process $( Y (\cdot), Z (\cdot), R (\cdot, \cdot) )$
is called a solution to the BSEEJ \eqref{bseej}, if $( Y (\cdot), Z (\cdot), R (\cdot, \cdot) ) \in
M_{\mathscr F}^2 ( 0, T; V ) \times M_{\mathscr F}^2 ( 0, T; H ) \times M_{\mathscr F}^{\nu, 2}(0,T; H)$ and, for every $\phi \in V$,  it holds that
\begin{eqnarray}\label{eq:c5}
( Y (t),  \phi )_H  &=& (\xi, \phi)_H
- \int_t^T \Big\langle  A (s) Y (s) +B(s)Z(s)+F ( s, Y (s), Z (s), R(s, \cdot) ), \phi \Big\rangle d t \nonumber \\
&& - \int_t^T ( Z (s), \phi )_H d W (s) -\int_t^T\int_E (R(s,e), \phi)_H\tilde \mu(ds,de), \quad \text{for a.e.
$t \in {\cal T}$, a.s.,}
\end{eqnarray}
or alternatively, $( Y (\cdot), Z (\cdot), R (\cdot, \cdot) )$ satisfies the following It\^{o}'s equation in $V^*$:
\begin{eqnarray}
Y(t)&=& \xi -\int_t^T \big[A (s) Y (s) d s+ B (s) Z (s)+F ( t, {Y} (s), {Z} (s),  R(s,\cdot) ) \big] d s \nonumber \\
&&- \int_t^T Z (s) d W (s)-\int_t^T \int_E R (s,e) d \tilde \mu(ds,de) , \quad t \in {\cal T} .
\end{eqnarray}
\end{defn}

\begin{thm}[{\bf Continuous Dependence Theorem}]\label{lem:1.4}
If $( Y (\cdot), Z (\cdot), R (\cdot, \cdot) )$ is the solution to the BSEEJ \eqref{bseej}
corresponding to the generator $( A, B, F, \xi )$, then the following estimate holds:
\begin{eqnarray}\label{eq:2.15}
&& {\mathbb E} \bigg [ \sup_{0 \leq t \leq T} \| Y (t) \|^2_H \bigg ]
+ {\mathbb E} \bigg [ \int_0^T \| Y (t) \|_V^2 d t \bigg ]
+ {\mathbb E} \bigg [ \int_0^T \| Z (t) \|^2_H d t \bigg ]+ {\mathbb E} \bigg [ \int_0^T \int_E \| R (t,e) \|^2_H  \nu (de)dt \bigg ] \nonumber \\
&& \qquad\qquad\qquad\qquad\ \leq K \bigg \{ {\mathbb E} [ \| \xi \|^2_H ] + {\mathbb E} \bigg
[ \int_0^T \| F ( t, 0, 0, 0 ) \|^2_H d t \bigg ] \bigg \} ,
\end{eqnarray}
where $K \triangleq K ( T, C, \alpha, \lambda )$ is a positive constant depending only on $T$, $C$, $\alpha$, and $\lambda$.
Moreover, if $( {\bar Y} (\cdot), {\bar Z} (\cdot), {\bar R} (\cdot, \cdot) )$ is a solution to the
BSEEJ \eqref{bseej} corresponding to another generator $(A, B, {\bar F}, {\bar \xi})$,
then we have
\begin{eqnarray}\label{eq:6.10}
&& {\mathbb E} \bigg [ \sup_{0 \leq t \leq T} \| Y (t) - {\bar Y} (t) \|_H^2 \bigg ]
+ {\mathbb E} \bigg [ \int_0^T \| Y (t) - {\bar Y} (t) \|_V^2 d t \bigg ]+
 {\mathbb E} \bigg [ \int_0^T \| Z (t) - {\bar Z} (t) \|^2_H d t\bigg] \nonumber \\
&&\quad \quad+ {\mathbb E} \bigg [ \int_0^T \int_E\| R (t,e) - {\bar R} (t,e) \|^2_H \nu (de)d t \bigg ] \\
&& \leq K \bigg \{ {\mathbb E} [ \| \xi - {\bar \xi} \|_H^2 ]
+ {\mathbb E} \bigg [ \int_0^T \| F ( t, \bar Y (t), \bar Z (t), \bar R(t, \cdot) ) - \bar F ( t, \bar Y (t), \bar Z (t),
\bar R(t, \cdot) ) \|^2_H d t \bigg ] \bigg \} . \nonumber
\end{eqnarray}
\end{thm}

\begin{proof}
If we take the generator $( A, B, {\bar F}, {\bar \xi} ) = (A, B, 0, 0)$,
then the corresponding solution to the BSEEJ \eqref{bseej} is $(\bar Y(\cdot),\bar Z(\cdot), \bar R(\cdot, \cdot)) = (0, 0, 0)$. Hence, the estimate \eqref{eq:2.15} follows from  \eqref{eq:6.10} immediately.
Therefore, it suffices to prove \eqref{eq:6.10}. To simplify our notations, we denote by
\begin{eqnarray*}
& {\hat Y} (t) \triangleq Y (t) - {\bar Y} (t) , \quad
{\hat Z} (t) \triangleq Z (t) - {\bar Z} (t) , \quad
{\hat R} (t,e) \triangleq R (t,e) - {\bar R} (t,e), \quad
{\hat \xi} \triangleq \xi - {\bar \xi} , \\
& {\hat F} (t) \triangleq F ( t, {\bar Y} (t), {\bar Z} (t), \bar R(t,\cdot))- {\bar F} (t, {\bar Y} (t), {\bar Z} (t), \bar R(t,\cdot)) , \quad
{\tilde F} (t) \triangleq F ( t, {Y} (t), {Z} (t), R(t,\cdot))- {\bar F} (t, {\bar Y} (t), {\bar Z} (t), \bar R(t,\cdot)) .
\end{eqnarray*}
From Lemma  \ref{lem:c1}, we obtain
\begin{equation}\label{Ito_Y}
\begin{split}
||\hat Y (t)||_H^2 =&||\hat \xi||^2- 2\int_t^T\langle A(s)\hat Y(s)+B(s)\hat Z(s)+\tilde F(s)+\hat F(s), \hat Y(s) \rangle
ds -2\int_t^T ( \hat Z(s), \hat Y(s) )_H dW(s) \nonumber \\
&- \int_t^T||\hat Z(s)||_H^2ds-\int_t^T\int_E\big[ ||\hat R(s,e)||_H^2+ 2 (\hat Y(s), \hat R(s,e))_H \big]
\tilde \mu(de,ds)-\int_t^T\int_E ||\hat R(s,e)||_H^2\nu(de)ds.
\end{split}
\end{equation}
Using the inequality $2
a b \leq \frac{1}{\varepsilon} a^2 + \varepsilon b^2$, $\forall a, b > 0$, $\varepsilon > 0$, we have
$$
-2\langle B(s)\hat Z(s),\hat Y(s)\rangle=-2\langle \hat Z(s),B^*\hat Y(s)\rangle_H\le \frac{1}{1+\varepsilon_1}\|Z(s)\|_H^2+(1+\varepsilon_1)\|B^*(s)\hat Y(s)\|^2,
$$
$$
-2\langle \tilde F(s),\hat Y(s)\rangle \le \varepsilon_2 \|\tilde F(s)\|^2_H+\frac{1}{\varepsilon_2} \|Y(s)\|_H^2\le C\varepsilon_1(\|\hat Y(s)\|_V^2+\|\hat Z(s)\|^2_H+\|\hat R(s,\cdot)\|^2_{M^{\nu,2}(E;H)})+\frac{1}{\varepsilon_2} \|Y(s)\|_H^2,
$$
and
$$
-2\langle \hat F(s),\hat Y(s)\rangle \le \|\hat Y(s)\|_H^2+\|\hat F(s)\|_H^2.
$$
Then, it implies that
\begin{align*}
||\hat Y (t)||_H^2 \le & \ ||\hat \xi||^2+\int_t^T \bigg(-2\langle A(s)\hat Y(s),\hat Y(s)\rangle+(1+\varepsilon_1)\|B^*(s)Y(s)\|_H^2+C\varepsilon_2 \|\hat Y(s)\|_V^2+(\frac{1}{\varepsilon_2} +1)\|\hat Y(s)\|_H^2\\
&+\|\hat F(s)\|_H^2-(1-\frac{1}{1+\varepsilon_1}-C\varepsilon_2)\|\hat Z(s)\|_H^2-(1-C\varepsilon_2) \|\hat R(s,\cdot)\|^2_{M^{\nu,2}(E;H)}\bigg)ds\\
&-2\int_t^T ( \hat Z(s), \hat Y(s) )_H dW(s)-\int_t^T\int_E\big[ ||\hat R(s,e)||_H^2+ 2 (\hat Y(s), \hat R(s,e))_H \big]
\tilde \mu(de,ds).
\end{align*}
From Assumption \ref{ass:6.1}, we have that
\begin{align*}
-2\langle A(s)\hat Y(s),\hat Y(s)\rangle+(1+\varepsilon_1)\|B^*(s)Y(s)\|_H^2\le& \ 2\varepsilon_1\langle A(s)\hat Y(s),\hat Y(s)\rangle+\lambda (1+\varepsilon_1) \|\hat Y(s)\|_H^2-\alpha (1+\varepsilon_1)\|\hat Y(s)\|_V^2\\
\le& \ (\lambda(1+\varepsilon_1))\|\hat Y(s)\|_H^2- (\alpha+\alpha\varepsilon_1-2C\varepsilon_1)\|\hat Y(s)\|_V^2.
\end{align*}
Hence, we get
\begin{align}
&||\hat Y (t)||_H^2+\int_t^T \bigg((1-\frac{1}{1+\varepsilon_1}-C\varepsilon_2)\|\hat Z(s)\|_H^2+(1-C\varepsilon_2) \|\hat R(s,\cdot)\|^2_{M^{\nu,2}(E;H)}\bigg)ds\nonumber\\
&\le \int_t^T\bigg((1+\frac{1}{\varepsilon_2}+\lambda(1+\varepsilon_1))\|\hat Y(s)\|_H^2-(\alpha+\alpha\varepsilon_1-2C\varepsilon_1-C\varepsilon_2)\|\hat Y(s)\|_V^2+\|\hat F(s) \|_H^2\bigg)ds\nonumber\\
&\quad+2\int_t^T ( \hat Z(s), \hat Y(s) )_H dW(s)+\int_t^T\int_E\big[ ||\hat R(s,e)||_H^2+ 2 (\hat Y(s), \hat R(s,e))_H \big]
\tilde \mu(de,ds).
\end{align}
From the integrability condition of the solution in Definition  \ref{defn:6.1},  we know that $\int_{0}^{\cdot} ( u(s), v^k(s)) d W^k_{s}$ is a uniformly integrable martingale.
Moreover, $\int_0^{\cdot}\int_E [ ||\hat R(s,e)||_H^2+ 2 (\hat Y(s), \hat R(s,e))_H ] \tilde \mu(de,ds)$ is also a
uniformly integrable martingale. Taking expectations on both sides, we have
\begin{align*}
&\mathbb E\left[ ||\hat Y (t)||_H^2+\int_t^T \bigg( (1-\frac{1}{1+\varepsilon_1}-C\varepsilon_2)\|\hat Z(s)\|_H^2+(1-C\varepsilon_2) \|\hat R(s,\cdot)\|^2_{M^{\nu,2}(E;H)} \bigg) ds \right]\\
&\le \mathbb E \left[ \int_t^T \bigg((1+\frac{1}{\varepsilon_2}+\lambda(1+\varepsilon_1))\|\hat Y(s)\|_H^2-(\alpha+\alpha\varepsilon_1-2C\varepsilon_1-C\varepsilon_2)\|\hat Y(s)\|_V^2+\|\hat F(s) \|_H^2\bigg) ds \right].
\end{align*}
Choosing sufficiently small $\varepsilon_1$ and $\varepsilon_2$ such that $1-\frac{1}{1+\varepsilon_1}-C\varepsilon_2 >0, 1-C\varepsilon_2>0$, and $\alpha+\alpha \varepsilon_1-2C\varepsilon_1-C\varepsilon_2>0$, we finally get that
\begin{eqnarray}\label{leq:c4}
&& {\mathbb E} [ \| {\hat Y} (t) \|_H^2 ]
+ {\mathbb E} \bigg [ \int_t^T \| {\hat Y} (s) \|_V^2 d s \bigg ]
+ {\mathbb E} \bigg [ \int_t^T \| {\hat Z} (s) \|_H^2 d s \bigg ]+{\mathbb E} \bigg [ \int_t^T
\int_E\| {\hat R} (s,e) \|_H^2 \nu(de)d s \bigg ]  \nonumber \\
&& \leq  K(\lambda, \alpha)  \bigg \{ \mathbb E [ \| {\hat \xi} \|_H^2 ]
+ \mathbb E \bigg [ \int_t^T \| {\hat F} (s) \|_H^2 d s \bigg ]
+ \mathbb E \bigg [ \int_t^T \| {\hat Y} (s) \|_H^2 d s \bigg ] \bigg \} .
\end{eqnarray}
Then using Gr\"onwall's inequality to \eqref{leq:c4} gives
\begin{eqnarray}\label{leq:c5}
&& \sup_{0 \leq t \leq T} {\mathbb E} [ \| {\hat Y} (t) \|_H^2 ]
+ {\mathbb E} \bigg [ \int_0^T \| {\hat Y} (t) \|_V^2 d t \bigg ]
+ {\mathbb E} \bigg [ \int_0^T \| {\hat Z} (t) \|_{H}^{2} d t \bigg ]
+{\mathbb E} \bigg [ \int_0^T \int_E\| {\hat R} (s,e) \|_H^2 \nu(de)d s \bigg ]   \nonumber \\
&& \leq  K(\alpha, \lambda ) \bigg \{ {\mathbb E} [ \| {\hat \xi} \|_H^2 ]
+ {\mathbb E} \bigg [ \int_0^T \| {\hat F} (t) \|_H^2 d t \bigg ] \bigg \} .
\end{eqnarray}
Using \eqref{Ito_Y} and the Burkholder-Davis-Gundy inequality
yields
\begin{align}\label{eq:2.29}
{\mathbb E} \bigg [ \sup_{0 \leq t \leq T} \| {\hat Y} (t) \|_H^2 \bigg ]
\leq& \  K(\alpha, \lambda ) \bigg \{ {\mathbb E} [ \| {\hat \xi} \|_H^2 ]
+ {\mathbb E} \bigg [ \int_0^T \| {\hat Y} (t) \|_H^2 d t \bigg ]
+ {\mathbb E} \bigg [ \int_0^T \| {\hat F} (t) \|_H^2 d t \bigg ] \bigg \} \nonumber \\
&+ 2 {\mathbb E} \bigg [ \sup_{0 \leq t \leq T} \bigg | \int_t^T (
{\hat Y} (s), {\hat Z} (s) )_H d W (s) \bigg | \bigg ] \nonumber
\\ & + {\mathbb E} \bigg [ \sup_{0 \leq t \leq T} \bigg | \int_t^T\int_E\big( ||\hat R(s,e)||_H^2
+ 2 (\hat Y(s), \hat R(s,e))_H   \big) \tilde \mu(de,ds) \bigg | \bigg ] \nonumber \\
\leq & \ K(\alpha, \lambda ) \bigg \{ {\mathbb E} [ \| {\hat \xi} \|_H^2 ] + {\mathbb
E} \bigg [ \int_0^T \| {\hat Y} (t) \|_V^2 d t \bigg ] + {\mathbb E}
\bigg [ \int_0^T \| {\hat F} (t) \|_H^2 d t \bigg ] \bigg \} \nonumber \\
& + \frac{1}{2} {\mathbb E} \bigg [ \sup_{0 \leq t \leq
T} \| {\hat Y} (t) \|_H^2 \bigg ] + K {\mathbb E} \bigg [ \int_0^T
\| {\hat Z} (t) \|_H^2 d t \bigg ]+K{\mathbb E} \bigg [ \int_t^T
\int_E\| {\hat R} (s,e) \|_H^2 \nu(de)d s \bigg ],
\end{align}
where the last inequality is obtained due to the fact that
\begin{equation*}
\begin{split}
{\mathbb E} \left [ \sup_{0 \leq t \leq T} \bigg | \int_t^T (
{\hat Y} (s), {\hat Z} (s) )_H d W (s)\bigg| \right ] &\le C\mathbb E\left[   \left( \int_0^T | (
{\hat Y} (s), {\hat Z} (s) )_H|^2ds \right)^{1/2}\right]\\
&\le C\mathbb E\left[\sup_{0 \leq s \leq T} \|\hat Y(s)\|_H \left(\int_0^T |Z(s)|_H^2ds  \right)^{1/2}  \right] \\
&\le C \mathbb E\left[\varepsilon \sup_{0 \leq s \leq T} \|\hat Y(s)\|^2_H+\frac{1}{\varepsilon} \int_0^T |Z(s)|_H^2ds   \right]
\end{split}
\end{equation*}
and
\begin{equation*}
\begin{split}
&{\mathbb E} \bigg [ \sup_{0 \leq t \leq T} \bigg | \int_t^T\int_E\big( ||\hat R(s,e)||_H^2
+ 2 (\hat Y(s), \hat R(s,e))_H \big) \tilde \mu(de,ds) \bigg | \bigg ] \\
&\le C \mathbb E \left[\int_0^T\int_E \bigg|||\hat R(s,e)||_H^2
+ 2 (\hat Y(s), \hat R(s,e))_H \bigg|   \nu(de)ds\right]\\
&\le C\mathbb E\left[ \varepsilon \sup_{0 \leq s \leq T} \|\hat Y(s)\|_H^2 +(1+\frac{1}{\varepsilon})\int_0^T \int_E \|R(s,e)\|_H^2\nu(de)ds \right] .
\end{split}
\end{equation*}
 Combining \eqref{eq:2.29} with \eqref{leq:c5} gives the desired result \eqref{eq:6.10}.
\end{proof}

\begin{thm}[{\bf Existence and uniqueness theorem of BSEEJ}]\label{thm:6.1}
Given a generator $(A, B, F, \xi)$ satisfying Assumption \ref{ass:6.1},
the BSEEJ \eqref{bseej} has a unique solution $( Y (\cdot), Z (\cdot), R (\cdot, \cdot) )
\in S_{\mathscr F}^2 ( 0, T; V ) \times M_{\mathscr F}^2 ( 0, T; H )
\times M_{\mathscr F}^{\nu, 2}(0,T; H)$.
\end{thm}

To prove this result, we first consider a simple case when $F$ is independent of $(Y,Z,R)$. To be more precise,
we consider a BSEEJ with $F$ replaced by a $V^*$-valued ${\mathbb F}$-predictable process $F_0$ as follows:
\begin{eqnarray}\label{bseej_linear}
Y(t)=\xi-\int_t^T [A(s)Y(s)+B(s)Z(s)+F_0(s)]ds-\int_t^TZ(s)dW(s)
-\int_t^T\int_{E}R(s,e)\tilde \mu (de,ds).
\end{eqnarray}
Now we state the result of the existence and uniqueness of a solution to the BSEEJ \eqref{bseej_linear}.

\begin{lem} \label{lem:6.4}
Suppose that the coefficients $(A, B, F_0,\xi)$ satisfy Assumption \ref{ass:6.1}.
Then the BSEEJ \eqref{bseej_linear} has a unique solution in the sense of Definition \ref{defn:6.1}.
\end{lem}

\begin{proof}
First of all, we fix a standard complete orthogonal basis $\{e_{i} | i=1,2,3,\dots\}$ in the space $H$ which is dense  in the space $V$.
For any $n,$ consider the following
finite-dimensional backward stochastic differential
equation in $\mathbb R^n$:
\begin{align}\label{eq:6.18}
y_i^n(t)=& \ (\xi, e_i) - \int_t^T \bigg ( \sum_{j=1}^ny_j^n(s) \langle A (s) e_j, e_i \rangle
-\sum_{j=1}^n z_j^n(s) \langle B(s)e_j, e_i\rangle- (F_0(s), e_i)_H \bigg ) d s \nonumber \\
&-\int_t^T z_i^n (s) d W (s)+ \int_t^T \int_E r_i^n(s,e)\tilde \mu(ds,de) , \quad \ i = 1, 2, \cdots, n .
\end{align}
%
Under Assumption \ref{ass:6.1}, from the existence and uniqueness theory for
the finite-dimensional BSDE with jumps, the above equation admits a unique  strong solution
$(y^n(\cdot), z^n(\cdot), r^n(\cdot, \cdot))$ such that
\begin{eqnarray*}
(y^n(\cdot), z^n(\cdot), r^n(\cdot, \cdot))
\in S_{\mathscr F}^2 ( 0, T; \mathbb R^n )\times M_{\mathscr F}^2 ( 0, T; \mathbb R^n )\times M_{\mathscr F}^{\nu,2 }( 0, T; \mathbb R^n ),
\end{eqnarray*}
where
$y^n(\cdot)=(y_1^n(\cdot),\cdots,  y_n^n(\cdot))$, $z^n(\cdot)=(z_1^n(\cdot),\cdots,  z_n^n(\cdot))$, and $r^n(\cdot)=(r_1^n(\cdot),\cdots,  r_n^n(\cdot))$.

Now we can define an  approximation solution to \eqref{bseej_linear} as follows:
$Y^n(\cdot):=\sum_{i=1}^{n}y_{i}^{n}(\cdot)e_{i}$,
$Z^{n}(\cdot):=
\sum_{i=1}^{n}z^{n}_{i}(\cdot)e_{i}$,  $F^n(\cdot)=\sum_{i=1}^n( F(\cdot),e_{i})_H e_i$,
$R^{n}(\cdot):=
\sum_{i=1}^{n}r^{n}_{i}(\cdot)e_{i}$,
and $\xi^n:=\sum_{i=1}^{n} (\xi,e_{i})_He_i$.
Then, from Equation \eqref{eq:6.18}, we see that
\begin{eqnarray}\label{eq:6.21}
(Y^{n}(t),e_{i})_{H} &=&\left( \xi^n,e_{i}\right)_{H}-\int_{t}^{T}\bigg(\left\langle A( s) Y^{n}(s),e_{i}\right\rangle
+\left\langle B\left( s\right) Y^{n}(s),e_{i}\right\rangle +( F^n(s),e_{i})_H\bigg)ds \nonumber \\
&&-\int_{t}^{T}\left( Z^n(s),e_{i}\right) _{H}dW(s)-\int_t^T \int_{E}\left( R^{n}\left( s,e\right) ,e_{i}\right) _{H}\tilde{\mu}(ds,de), \quad i=1,2,\cdots, n .
\end{eqnarray}
Now applying It\^{o} formula to $\|Y^n(t)\|^2_H$, we get
\begin{eqnarray}\label{eq:6.2}
||Y^n (t)||_H^2&=&||\xi^n||^2-2\int_t^T\langle A(s)Y^n(s)+B(s)Z^n(s)+ F^n(s),  Y^n(s) \rangle ds \nonumber \\
&& - \int_t^T|| Z^n(s)||_H^2ds
-\int_t^T\int_E || R^n(s,e)||_H^2\nu(de)ds -2\int_t^T\langle  Z^n(s),  Y^n(s) \rangle dW(s) \nonumber \\
&&
-\int_t^T\int_E\big( ||R^n(s,e)||_H^2+ 2 (Y^n(s), R^n(s,e))_H  \big) \tilde \mu(ds,de) .
\end{eqnarray}
Therefore, under Assumption \ref{ass:6.1}, similar to the proof of the estimate \eqref{eq:6.10},  using Gr\"onwall's inequality
and the Burkholder-Davis-Gundy inequality, we can easily get the following estimate:
\begin{eqnarray}\label{eq:6.23}
&& {\mathbb E} \bigg [ \sup_{0 \leq t \leq T} \| Y^n (t) \|^2_H \bigg ]
+ {\mathbb E} \bigg [ \int_0^T \| Y^n (t) \|_V^2 d t \bigg ]
+ {\mathbb E} \bigg [ \int_0^T \| Z^n (t) \|^2_H d t \bigg ]+ {\mathbb E} \bigg [ \int_0^T\int_E \| R^n (t,e) \|^2_H  \nu (de)dt \bigg ] \nonumber \\
&& \qquad\qquad\qquad\qquad \leq K \bigg \{ {\mathbb E} [ \| \xi^n \|^2_H ] + {\mathbb E} \bigg
[ \int_0^T \| F^n( t) \|^2_H d t \bigg ] \bigg \} \leq K \bigg \{ {\mathbb E} [ \| \xi \|^2_H ] + {\mathbb E} \bigg
[ \int_0^T \| F ( t) \|^2_H d t \bigg ] \bigg \} .
\end{eqnarray}
This inequality implies that there is a subsequence $\{n^\prime\}$ of $\left\{ n\right\} $ and a triplet $( Y (\cdot), Z (\cdot), R(\cdot, \cdot) )
\in M_{\mathscr F}^2 ( 0, T; V )\times M_{\mathscr F}^2 ( 0, T; H )\times M_{\mathscr F}^{\nu,2 }( 0, T; \ H) $ such that
$Y^{n^{\prime }}\rightarrow Y$ weakly in $M_{\mathscr F}^2 ( 0, T; V )$,
$Z^{n^{\prime }} \rightarrow  Z$ weakly in $M_{\mathscr F}^2 ( 0, T; H )$,
and $R^{n^{\prime }} \rightarrow  R$ weakly in $M_{\mathscr F}^{\nu,2} ( 0, T; H )$, respectively.
Let $\Pi$ be an arbitrary bounded random variable on $(\Omega,\mathscr F)$ and
$\psi$ be an arbitrary bounded measurable function on $[0,T]$. From  the equality \eqref{eq:6.21}, for any $n\in \mathbb{N}^*$
and basis $e_{i}$, where $i\leq n$, we have
\begin{eqnarray}\label{eq:6.25}
&& \mathbb E \bigg [ \int_0^T \Pi \psi(t) (Y^{n}(t),e_{i})_{H}dt \bigg ] \nonumber \\
&& = \mathbb E \bigg [ \int_0^T \Pi \psi(t) \bigg\{\left( \xi^n,e_{i}\right)
_{H}-\int_{t}^{T}\big(\left\langle A\left( s\right) Y^{n}(s)
,e_{i}\right\rangle +\left\langle B\left( s\right) Y^{n}(s)
,e_{i}\right\rangle +( F(s),e_{i})_H\big)ds \nonumber \\
&& \quad -\int_{t}^{T}\left( Z^n(s),e_{i}\right) _{H}dW(s)-\int_t^T \int_E \left( R^{n}\left( s,e\right) ,e_{i}\right) _{H}\tilde{\mu}(ds,de)\bigg\}dt \bigg ] .
\end{eqnarray}
Now let $n'\rightarrow \infty$ on the both sides of the above equation to get its limit. Firstly, from the weak convergence
property of $\{Y^n\}_{n=1}^\infty$ in $M_{\mathscr F}^2 ( 0, T; V )$, we have
\begin{eqnarray}\label{eq:6.26}
    \lim_{n'\rightarrow \infty}\mathbb E \bigg [ \int_{0}^{T}\Pi\psi(t)(Y^{n'}(t), e_i)_H dt \bigg ]
    &=&\lim_{n'\rightarrow \infty}\mathbb E \bigg [ \int_{0}^{T}\mathbb E[\Pi|\mathscr F_t]\psi(t)(Y^{n'}(t), e_i)_H dt \bigg ] \nonumber \\
    &=&\lim_{n'\rightarrow \infty}\mathbb E \bigg [ \int_{0}^{T}(Y^{n'}(t),\mathbb E[\Pi|\mathscr F_t]\psi(t)e_i)_H dt \bigg ]  \nonumber \\
    &=&\mathbb E \bigg [ \int_{0}^{T}(Y(t),\mathbb E[\Pi|\mathscr F_t]\psi(t)e_i)dt \bigg ] \nonumber \\
    &=&\mathbb E \bigg [ \int_{0}^{T}\Pi\psi(t)(Y(t),e_i)dt \bigg ] ,
\end{eqnarray}
and
\begin{eqnarray}
    \lim_{n'\rightarrow \infty}\mathbb E \bigg [ \int_{t}^{T}\Pi\langle A(s)Y^{n'}(s), e_i\rangle ds \bigg ]
    &=&\lim_{n'\rightarrow \infty}\mathbb E \bigg [ \int_{t}^{T}\mathbb E[\Pi|\mathscr F_s]\langle A(s)Y^{n'}(s), e_i\rangle ds \bigg ] \nonumber \\
    &=&\lim_{n'\rightarrow \infty}\mathbb E \bigg [ \int_{t}^{T}\langle A(s)Y^{n'}(s),\mathbb E[\Pi|\mathscr F_s]e_i\rangle ds \bigg ] \nonumber \\
    &=&\lim_{n'\rightarrow \infty}\mathbb E \bigg [ \int_{t}^{T}\langle Y^{n'}(s), A^*(s)\mathbb E[\Pi|\mathscr F_s]e_i\rangle ds \bigg ] \nonumber \\
    &=&\mathbb E \bigg [ \int_{0}^{T}\langle Y(s), A^*(s)\mathbb E[\Pi|\mathscr F_s]e_i\rangle ds \bigg ] \nonumber \\
    &=&\mathbb E \bigg [ \int_{0}^{T}\Pi\langle A(s)Y(s),e_i\rangle ds \bigg ],
\end{eqnarray}
where the orders of integration, expectation, and limit can be exchanged due to integrability of related processes. More precisely, in view of \eqref{eq:6.5} and \eqref{eq:6.23}, we conclude that
\begin{eqnarray*}
\mathbb E \bigg [ \bigg| \int_{t}^{T} \Pi \la A(s)Y^{n'}(s), e_i\rangle ds\bigg| \bigg ]
\leq C \bigg \{ \mathbb E \bigg [ \int_{0}^{T} ||Y^{n'}(s)||_V^2ds \bigg ] \bigg\}^{\frac{1}{2}} <C<\infty,
\end{eqnarray*}
where the constant $C$ is independent of $n'$. Hence from Fubini's Theorem and Lebesgue's Dominated Convergence Theorem, we have
\begin{eqnarray}\label{eq:6.28}
&& \mathbb E \bigg [ \int_{0}^{T}\Pi\psi(t)\int_{t}^{T}\langle A(s)Y^{n'}(s),e_i\rangle ds dt \bigg ] \nonumber \\
&& = \int_{0}^{T} \psi(t) \mathbb E \bigg [ \int_{t}^{T} \Pi\langle A(s)Y^{n'}(s),e_i\rangle ds \bigg ] dt
\rightarrow \int_{0}^{T} \psi(t) \mathbb E \bigg [ \int_{t}^{T} \Pi\langle A(s)Y(s),e_i\rangle ds \bigg ] dt .
\end{eqnarray}
Similarly, we have
\begin{eqnarray} \label{eq:6.29}
\mathbb E \bigg [ \int_{0}^{T}\Pi\psi(t)\int_{t}^{T} \langle B(s)Y^{n'}(s), e_i\rangle ds dt \bigg ]
\rightarrow \mathbb E \bigg [ \int_{0}^{T}\Pi\psi(t)\int_{t}^{T} \langle B(s)Y(s), e_i\rangle ds dt \bigg ] .
\end{eqnarray}

Since the stochastic integrals with respect to the Brownian motion $W$ and the Poisson random martingale measure
$\tilde \mu$ are linear and strongly continuous mappings from $M_{\mathscr F}^2 ( t, T; \mathbb R )\times M_{\mathscr F}^{\nu, 2} ( t, T; \mathbb R )$ to $L^2(\Omega,{\mathscr F}_T,P;\mathbb R)$, respectively,
they are also weakly continuous. Therefore, it follows from the weak convergence property of $Z^n$ and $R^n$ that
\begin{eqnarray}
&& \lim_{n' \rightarrow \infty} \mathbb E\bigg[\Pi \bigg(\int_{t}^{T}\left( Z^n(s),e_{i}\right) _{H}dW(s)
+\int_t^T \int_{E}\left( R^{n}\left( s,e\right) ,e_{i}\right) _{H}\tilde{\mu}(de,ds)\bigg)\bigg] \nonumber \\
&& = \mathbb E \bigg[\Pi \bigg(\int_{t}^{T}\left( Z(s),e_{i}\right) _{H}dW(s)
+\int_t^T \int_{E}\left( R\left( s,e\right) ,e_{i}\right) _{H}\tilde{\mu}(de,ds)\bigg)\bigg] .
 \end{eqnarray}
Moreover,
\begin{eqnarray}
&&\phi(t)\mathbb E\bigg[\Pi \bigg(\int_{t}^{T}\left( Z^n(s),e_{i}\right) _{H}dW(s)
+\int_t^T \int_{E}\left( R^{n}\left( s,e\right) ,e_{i}\right) _{H}\tilde{\mu}(de,ds)\bigg)\bigg] \nonumber \\
&&\leq  \frac{1}{2}\phi^2(t)\mathbb E |\Pi|^2
+ C \bigg \{ \mathbb E \bigg [ \int_0^T||Z^n(s)||^2_Hdt \bigg ]
+\mathbb E \bigg [ \int_0^T\int_E||R^n(s,e)||^2_H\nu(de)dt \bigg ] \bigg \} \leq C.
\end{eqnarray}
Hence, by Fubini's Theorem and Lebesgue's Dominated Convergence Theorem, we have
\begin{eqnarray}\label{eq:6.33}
&& \lim_{n'\rightarrow \infty} \mathbb E \bigg[ \int_0^T \phi(t)\Pi \bigg(\int_{t}^{T}\left( Z^n(s),e_{i}\right) _{H}dW(s)
+\int_t^T \int_{E}\left( R^{n}\left( s,e\right) ,e_{i}\right) _{H}\tilde{\mu}(de,ds)\bigg)\bigg]dt \nonumber \\
&&=\lim_{n'\rightarrow \infty} \int_0^T \phi(t) \mathbb E\bigg[\Pi \bigg(\int_{t}^{T}\left( Z^n(s),e_{i}\right) _{H}dW(s)
+\int_t^T \int_{E}\left( R^{n}\left( s,e\right) ,e_{i}\right) _{H}\tilde{\mu}(de,ds)\bigg)\bigg] d t \nonumber \\
&&=\int_t^T\phi(t)\mathbb E
      \bigg[\Pi \bigg(\int_{t}^{T}\left( Z(s),e_{i}\right) _{H}dW(s)
+\int_t^T \int_{E}\left( R\left( s,e\right) ,e_{i}\right) _{H}\tilde{\mu}(de,ds)\bigg)\bigg]
\nonumber \\
&&=\mathbb E\int_t^T
      \bigg[\phi(t)\Pi \bigg(\int_{t}^{T}\left( Z(s),e_{i}\right) _{H}dW(s)
+\int_t^T \int_{E}\left( R\left( s,e\right) ,e_{i}\right) _{H}\tilde{\mu}(de,ds)\bigg)\bigg] .
 \end{eqnarray}

Therefore, combining \eqref{eq:6.26},\eqref{eq:6.28}, \eqref{eq:6.29}, and \eqref{eq:6.33}, and letting $n'\rightarrow \infty$ in \eqref{eq:6.25},
we conclude that
\begin{eqnarray}\label{eq:6.34}
&& \mathbb E \bigg [ \int_0^T \Pi \psi(t) (Y(t),e_{i})_{H}dt \bigg ] \nonumber \\
&& = \mathbb E \bigg [ \int_0^T \Pi \psi(t) \bigg\{\left( \xi(t),e_{i}\right)_{H}
-\int_{t}^{T}\big(\left\langle A (s) Y (s),e_{i}\right\rangle
+\left\langle B( s) Y (s),e_{i}\right\rangle +( F(s),e_{i})_H\big)ds  \nonumber \\
&& \quad -\int_{t}^{T}\left( Z(s),e_{i}\right) _{H}dW(s)
-\int \int_{(t,T]\times E}\left( R\left( s,e\right) ,e_{i}\right) _{H}\tilde{\mu}(de,ds)\bigg\}dt \bigg ] .
\end{eqnarray}
This implies that for a.s. $\left( t,\omega \right) \in \lbrack 0,T]\times \Omega$,
\begin{eqnarray}\label{eq:6.34}
(Y(t),e_{i})_{H} &=& \left( \xi(t),e_{i}\right)_{H}-\int_{t}^{T}\big(\left\langle A\left( s\right) Y\left( t\right),e_{i}\right\rangle
+\left\langle B\left( s\right) Y\left( t\right),e_{i}\right\rangle+( F(s),e_{i})_H\big)ds \nonumber \\
&& -\int_{t}^{T}\left( Z(s),e_{i}\right) _{H}dW(s) -\int_t^T \int_{E}\left( R\left( s,e\right) ,e_{i}\right) _{H}\tilde{\mu}(de,ds),
\end{eqnarray}
thanks to the arbitrariness of $\Pi$ and $\psi(\cdot)$.
Since the standard complete orthogonal basis $\{e_{i} | i=1,2,3,\dots\}$ in $H$ is dense in the space $V$, it holds that for every $\phi \in V$ and
a.e. $( t, \omega ) \in [ 0, T ] \times \Omega$,
\begin{eqnarray}\label{eq:6.34}
(Y(t),\phi)_{H} &=& \left( \xi(t),\phi\right)_{H}-\int_{t}^{T}\big(\left\langle A\left( s\right) Y\left( t\right),\phi\right\rangle
+\left\langle B\left( s\right) Y\left( t\right),e_{i}\right\rangle+( F(s),\phi)_H\big)ds \nonumber \\
&& -\int_{t}^{T}\left( Z(s),\phi\right) _{H}dW(s) -\int \int_{(t,T]\times E}\left( R\left( s,e\right) ,\phi\right) _{H}\tilde{\mu}(de,ds) .
\end{eqnarray}
Therefore, from Definition \ref{defn:6.1}, we conclude that the triplet $( Y (\cdot), Z (\cdot), R(\cdot, \cdot) )$
is the solution to the BSEEJ \eqref{eq:6.17}. Thus the existence is proved. The uniqueness is an immediate result of Theorem \ref{lem:1.4}.
\end{proof}

\begin{proof}[Proof of Theorem \ref{thm:6.1}]
We first take an arbitrary process $F_0 (\cdot) \in M^2_{\mathscr F} ( 0, T; H )$.
Given that $\rho \in [0, 1]$, consider the following BSEEJ:
\begin{eqnarray}\label{eq:1.7}
Y(t) &=& \xi - \int_t^T \big[ A (s) Y (s) +B(s)Z(s) +\rho F ( s, Y (s), Z (s), R( s, e )) + F_0 (s) \big ] d s  \nonumber \\
&& -\int_t^T Z (s) d W (s) -\int_t^T\int_E R (s,e) d \tilde \mu(de, ds) .
\end{eqnarray}
Note that the coefficients $( A, B,\rho F + F_0, \xi )$ of the BSEEJ \eqref{eq:1.7} satisfy
Assumption \ref{ass:6.1} with the same constants $\alpha, \lambda, C$.  If we can prove
that the BSEEJ \eqref{eq:1.7} admits a unique solution for any $\rho$ and $F_0$, then setting $\rho = 1$ and $F_0 (\cdot) \equiv 0$
yields that the BSEEJ \eqref{eq:2.7} has a unique solution $( Y (\cdot), Z (\cdot), R(\cdot, \cdot) )
\in S_{\mathscr F}^2 ( 0, T; V ) \times M_{\mathscr F}^2 ( 0, T; H )\times M_{\mathscr F}^{\nu, 2} ( 0, T; H )$.

Suppose for some $\rho=\rho_0$, the BSEEJ \eqref{eq:1.7} has a unique solution $( Y (\cdot), Z (\cdot), R (\cdot, \cdot) )
\in M_{\mathscr F}^2 ( 0, T; V ) \times M_{\mathscr F}^2 (0, T; H)\times M_{\mathscr F}^{\nu,2} ( 0, T; H )$,
for any $F_0 (\cdot) \in M^2_{\mathscr F} (0, T; H)$. For another $\rho$, the BSEEJ \eqref{eq:1.7} can be rewritten as
\begin{eqnarray}\label{eq:1.8}
Y(t) &=& \xi - \int_t^T \big[ A (s) Y (s)+B(s)Z(s)+\rho_0 F ( s, Y (s), Z (s), R( s, e ) )+ F_0 (s)\nonumber \\
&&+(\rho-\rho_0) F( s, Y (s), Z (s), R( s, e ) )\big ] d s -\int_t^T Z (s) d W (s) -\int_t^T\int_E R (s,e) d \tilde \mu(de, ds) .
\end{eqnarray}
For any $( y (\cdot), z (\cdot), r (\cdot, \cdot) ) \in M_{\mathscr F}^2 ( 0, T; V) \times M_{\mathscr F}^2 ( 0, T; H )
\times M_{\mathscr F}^{\nu,2} ( 0, T; H )$, the following BSEEJ
\begin{eqnarray}\label{eq:1.9}
Y(t) &=& \xi - \int_t^T \big[ A (s) Y (s)+B(s)Z(s)+\rho_0 F ( s, Y (s), Z (s),R( s, e ) )+ F_0 (s)\nonumber \\
&& +(\rho-\rho_0) F ( s, y (s), z (s), r( s, e ) )\big ] d s -\int_t^T Z (s) d W (s)-\int_t^T\int_E R (s,e) d \tilde \mu(de, ds)
\end{eqnarray}
has a unique solution $( Y (\cdot), Z (\cdot), R (\cdot, \cdot) ) \in M_{\mathscr F}^2 ( 0, T; V) \times M_{\mathscr F}^2( 0, T; H)
\times M_{\mathscr F}^{\nu,2} ( 0, T; H )$. Thus, we can define a mapping ${\cal I}$ from $M_{\mathscr F}^2 ( 0, T; V) \times
M_{\mathscr F}^2( 0, T; H)\times M_{\mathscr F}^{\nu,2} ( 0, T; H )$ onto itself such that ${\cal I} ( y (\cdot), z (\cdot) ,r(\cdot,\cdot))
= ( Y (\cdot), Z (\cdot), R(\cdot,\cdot) )$. Moreover, we see that $(Y(\cdot),Z(\cdot),R(\cdot,\cdot))$ is a solution of \eqref{eq:1.8} if and only if it is a fixed point of $\mathcal I$.

For any $( y_i (\cdot), z_i (\cdot), r_i(\cdot, \cdot) ) \in M_{\mathscr F}^2 ( 0,
T; V ) \times M_{\mathscr F}^2 ( 0, T; H )\times M_{\mathscr F}^{\nu, 2} ( 0, T; H )$, we can find $( Y_i
(\cdot), Z_i (\cdot), R_i(\cdot, \cdot) ) \in M_{\mathscr F}^2 ( 0, T; V ) \times
M_{\mathscr F}^2 ( 0, T; H ) \times M_{\mathscr F}^{\nu, 2} ( 0, T; H )$ through the mapping ${\cal I} ( y_i
(\cdot), z_i (\cdot), r_i(\cdot) ) = ( Y_i (\cdot), Z_i (\cdot), R_i(\cdot) )$, for $i = 1, 2$.
According to the a priori estimate \eqref{eq:6.10} and the Lipschitz continuity of $F$, we have
\begin{eqnarray}
&& {\mathbb E} \bigg [ \int_0^T \| Y _1(t)-Y_2(t) \|_V^2 d t \bigg ]
+ {\mathbb E} \bigg [ \int_0^T \| Z_1 (t)-Z_2(t) \|^2_H d t \bigg ]+ {\mathbb E} \bigg [ \int_0^T \| R _1(t,e)-R_2(t,e) \|^2_H  \nu (de)dt \bigg ] \nonumber\\
&& \leq K \big|\rho - \rho_0 \big|^2\mathbb E \bigg[\int_0^T \bigg\|F ( t, y_1 (t), z_1 (t),r_1 ( t,\cdot))
-F ( t, y_2 (t), z_2 (t),r_2 ( t,\cdot)) \bigg\|^2_Hdt\bigg] \nonumber \\
&& \leq K | \rho - \rho_0 |^2 \times \bigg\{{\mathbb E} \bigg [ \int_0^T \| y _1(t)-y_2(t) \|_V^2 d t \bigg ]
+ {\mathbb E} \bigg [ \int_0^T \| z_1 (t)-z_2(t) \|^2_H d t \bigg ] \nonumber \\
&& \qquad\qquad\qquad\qquad\qquad + {\mathbb E} \bigg [ \int_0^T \| r _1(t,e)-r_2(t,e) \|^2_H  \nu (de)dt \bigg ] \bigg\},
\end{eqnarray}
where $K \triangleq K ( C, \lambda, \alpha, \rho_0 )$ is a constant independent of $\rho$.
If $| \rho - \rho_0 |< \frac{1}{2 \sqrt {K}}$, the mapping ${\cal I}$ is strictly contractive in
$M_{\mathscr F}^2 ( 0, T; V ) \times M_{\mathscr F}^2 ( 0, T; H)\times M_{\mathscr F}
^{\nu,2} ( 0, T; H)$, which admits a fixed point. Hence, it implies that
the BSEEJ \eqref{eq:1.7} with the coefficients $( A,
B, \rho F + F_0, \xi )$ admits a unique solution
$( Y (\cdot), Z (\cdot), R(\cdot, \cdot) ) \in M_{\mathscr F}^2 ( 0, T; V ) \times M_{\mathscr F}^2 ( 0, T; H )\times M_{\mathscr F}^{\nu,2} ( 0, T; H)$.
From Lemma \ref{lem:6.4} , the uniqueness and existence of a solution to the BSEEJ
\eqref{eq:1.7} is true for $\rho=0$. Then starting from $\rho = 0$, we have that the BSEEJ
\eqref{eq:1.7} also admits a unique solution for $\rho \in [ \frac{i - 1}{2 \sqrt {K}}, \frac{i}{2 \sqrt {K}} )$,
$i = 1, 2, \cdots$. Therefore, setting $i = [2 \sqrt {K}] + 1$ and $F_0 (\cdot) \equiv 0$ leads to
that the BSEEJ \eqref{eq:1.7} with the coefficients $( A, B, F, \xi )$, i.e., the BSEEJ \eqref{eq:2.7},
has a unique solution $( Y (\cdot), Z (\cdot),
 R(\cdot)) \in M_{\mathscr F}^2 ( 0, T; V ) \times
M_{\mathscr F}^2 ( 0, T; H )\times M_{\mathscr F}^{\nu,2} ( 0, T; H)$. Moreover, from the a priori estimate \eqref{eq:2.15},
we obtain $ Y (\cdot) \in S_{\mathscr F}^2 ( 0, T; V )$. This completes the proof.
\end{proof}

\subsection{Stochastic HJB Equation}

In this subsection, we recast the stochastic HJB equation as a BSEEJ and then establish the existence and
uniqueness result of a weak solution in the sense of the Sobolev space. We see that the super-parabolic condition
\eqref{eq:2.9} is crucial for BSEEJ theory. Due to the limitation of this approach,
we can only deal with the special case, in which the coefficient $\sigma$ does not contain the control variable $u$. In addition, we need a nondegeneracy assumption on $\sigma$.

Let
\begin{eqnarray}
\sigma(\cdot)=(\sigma_1(\cdot),\sigma_2(\cdot),\cdots,\sigma_{d-1}(\cdot), \sigma_d(\cdot))
\end{eqnarray}
and
$$
\hat\sigma(\cdot)=(\sigma_1(\cdot),\sigma_2(\cdot),\cdots,\sigma_{d-1}(\cdot)) .
$$
Let us define a sub-filtration ${\mathbb G} : = \{ {\mathscr G } (t) | t \in {\cal T} \}$ of ${\mathbb F}$,
which is a $\mathbb P$-augmentation of the natural filtration generated by the one-dimensional Brownian motion
$W_d(\cdot)$ and the Poisson random measure $\tilde \mu(\cdot, \cdot)$. We assume that all the coefficients involved
in Problem \ref{pro:2.1}  are restricted to ${\mathbb G}$-predictable processes
or ${\mathscr G}_T$-measurable. The cost functional is still defined as follows:
\begin{eqnarray}\label{eq:6.39}
\mathbb J(t,\xi;u(\cdot))= Y^{t,\xi;u}(t),
\end{eqnarray}
where $(Y^{t,\xi;u},Z^{t,\xi;u},K^{t,\xi;u})$ solves
\begin{equation*}
\begin{split}
dY(s)=&-f(s,X^{t,\xi;u}(s),u(s),Y(s),Z(s),\int_E K(s,e)l(s,e)\nu(de))ds\\
&+\sum_{i=1}^{d}Z_i(s)dW_i(s)+\int_E K(s,e)\tilde \mu(de,ds), \quad Y(T)=h(X^{t,\xi;u}(T)).
\end{split}
\end{equation*}
The corresponding stochastic HJB equation is formally given by the following form:
 \begin{eqnarray*}
\left\{
\begin{aligned}
-d V(t,x)=& \ \bigg\{
\frac{1}{2} \mbox{tr} \big [ \sigma\sigma^\top(t,x)D^2V(t,x) \big]
+ \sum_{i=1}^d\langle \sigma_i(t,x), D\Phi_i(t,x)\rangle
+\inf_{u\in
	U}\bigg [\langle  b(t,x,u), DV(t,x)\rangle \\
&+f(t,x,u,V,\sigma^TDV+\Phi,\int_E\left( \mathcal I V(t,e,x,u)+\Psi(t,e,x+g)\right)l(t,e)\nu(de)) \\
&+\displaystyle\int_
{E}\big[\mathcal I V(t,e,x,u)-(g(t, e,x,u), D V(t,x))\big]\nu(d e)+\displaystyle\int_{E}\displaystyle \mathcal I\Psi(t,e,x,u)\nu(d e)\bigg]\bigg\}dt\\
&-\sum_{i=1}^d\Phi_i(t,x)dW_i(t)
-\displaystyle
\int_{E}\Psi( t,e,x)\tilde \mu (d e,dt),\\
V(T,x)=& \ h(x).
\end{aligned}\right.
\end{eqnarray*}
Since the randomness of the coefficients only comes from $W_d$ and $\tilde \mu$, it holds that $\Phi_i=0$ for $i=1,2,\cdots,d-1$. Thus, the above stochastic HJB equation reduces to
\begin{eqnarray}\label{eq:c41}
 \left\{
 \begin{aligned}
 -d V(t,x)=& \ \bigg\{
 \frac{1}{2} \mbox{tr} \big [ \sigma\sigma^\top(t,x)D^2V(t,x) \big]
 + \langle \sigma_d(t,x), D\Phi_d(t,x)\rangle
 +\inf_{u\in
 	U}\bigg [\langle  b(t,x,u), DV(t,x)\rangle \\
 &+f(t,x,u,V,\sigma^TDV+\Phi_d,\int_E\left( \mathcal I V(t,e,x,u)+\Psi(t,e,x+g)\right)l(t,e)\nu(de)) \\
 &+\displaystyle\int_
 {E}\big[\mathcal I V(t,e,x,u)-(g(t, e,x,u), D V(t,x))\big]\nu(d e)+\displaystyle\int_{E}\displaystyle \mathcal I\Psi(t,e,x,u)\nu(d e)\bigg]\bigg\}dt\\
 &-\Phi_d(t,x)dW_d(t)
 -\displaystyle
 \int_{E}\Psi( t,e,x)\tilde \mu (d e,dt),\\
 V(T,x)=& \ h(x).
 \end{aligned}\right.
\end{eqnarray}
Indeed, this is a Cauchy problem for semi-martingale backward stochastic partial differential equations in non-divergence form. Next we rewrite this equation in divergence form. Note
\begin{eqnarray*}
&\mbox{tr}[\sigma\sigma^\top D^2V(t,x)]= \nabla \cdot[\sigma\sigma^\top DV(t,x)]- \langle  \nabla \cdot(\sigma\sigma^\top)(t,x),  DV(t,x)\rangle , \\
&\langle \sigma_d, D\Phi_d(t,x)\rangle= \nabla  \cdot[\Phi_d \sigma_d ]-\Phi_d\nabla \cdot \sigma_d,
\end{eqnarray*}
where (with $\sigma=(\sigma_1, \cdots, \sigma_d),$ each $\sigma_i$ takes values in $\mathbb R^n$)
\begin{eqnarray*}
\nabla\cdot \sigma = (\nabla\cdot \sigma_1, \cdots, \nabla\cdot \sigma_d)^\top .
\end{eqnarray*}
In view of the above reduction, we have the following divergence form of the BSPDE (i.e., stochastic HJB equation):
\begin{eqnarray}\label{eq:6.17}
  \left\{
  \begin{aligned}
    -d V(t,x)=& \ \bigg\{\displaystyle
    \frac{1}{2}\nabla \cdot[\sigma \sigma^\top(t,x)DV(t,x)]
    + \nabla  \cdot[\Phi_d(t,x) \sigma_d(t,x)]-\langle  \nabla \cdot[\sigma\sigma^\top(t,x)],  DV(t,x) \rangle
    \\
    &-\nabla \cdot [\sigma_d(t,x)\Phi_d(t,x)]+\inf_{u\in
    	U}\bigg [\langle  b(t,x,u), DV(t,x)\rangle \\
    &+f(t,x,u,V,\sigma^TDV+\Phi_d,\int_E\left( \mathcal I V(t,e,x,u)+\Psi(t,e,x+g)\right)l(t,e)\nu(de)) \\
    &+\displaystyle\int_
    {E}\big[\mathcal I V(t,e,x,u)-(g(t, e,x,u), D V(t,x))\big]\nu(d e)+\displaystyle\int_{E}\displaystyle \mathcal I\Psi(t,e,x,u)\nu(d e)\bigg]\bigg\}dt\\
    &-\Phi_d(t,x)dW_d(t)
    -\int_{E}\Psi( t,e,x)\tilde \mu (d e,dt),\\
     V(T,x)=& \ h(x).
\end{aligned}
\right.
\end{eqnarray}
The following
definition gives the generalized weak solution to Eq. \eqref{eq:c41} or \eqref{eq:6.17}.

\begin{defn}\label{defn:6.2}
A triplet $(V, \Phi, \Psi)\in M_{\mathscr F}^2 ( 0, T; V )\times M_{\mathscr F}^2 ( 0, T; H )\times M_{\mathscr F}^{\nu,2 }( 0, T; H)$
is called an adapted weak solution to \eqref{eq:c41} or \eqref{eq:6.17} if,  for every $\phi \in H^1(\mathbb R^n)$ and a.e. $( t,
\omega) \in [ 0, T ] \times \Omega $, it holds that
\begin{eqnarray}\label{eq:6.48}
 \int_{\mathbb R^n}V(t,x)\phi(x)dx &=& \int_{\mathbb R^n}h(x)\phi(x)+ \int_t^T\int_{\mathbb R^n}\bigg\{
   -\bigg\langle \frac{1}{2}\sigma\sigma^\top(s,x)DV(s,x)+ \sigma_d(s,x)\Phi_d(s,x) , D\phi(x)\bigg\rangle \nonumber \\
    &&+\bigg[-
    \Big\langle \nabla \cdot[\sigma\sigma^\top(s,x)],  DV(s,x) \Big\rangle-\Phi_d(s,x)\nabla \cdot \sigma_d(s,x)+\displaystyle+\inf_{u\in
    	U}\bigg [\langle  b(s,x,u), DV(s,x)\rangle\nonumber \\
    &&+f(s,x,u,V,\sigma^TDV+\Phi_d,\int_E\left( \mathcal I V(s,e,x,u)+\Psi(s,e,x+g)\right)l(s,e)\nu(de)) \nonumber\\
    &&+\displaystyle\int_
    {E}\big[\mathcal I V(s,e,x,u)-(g(s, e,x,u), D V(s,x))\big]\nu(d e)+\displaystyle\int_{E}\displaystyle \mathcal I\Psi(s,e,x,u)\nu(d e)\bigg]\bigg\}dxds \nonumber \\
    &&-\int_t^T\int_{\mathbb R^n}\Phi_d(s,x)\phi(x)dW_d(s)
    -\int_t^T\int_{E}\int_{\mathbb R^n}
    \Psi( s,e,x)\phi(x)dx\tilde\mu (d e,ds) .
 \end{eqnarray}
\end{defn}

\begin{ass}\label{ass:6.2}
The diffusion coefficient $\hat \sigma$ is uniformly positive:
\begin{equation}\label{ieq:non-degen}
 \hat \sigma \hat \sigma^\top(t,x)\geq 2\alpha I, \quad \forall(t,x)\in [0, T]\times \mathbb R^n,
\end{equation}
where $\alpha$ is a  positive constant.
\end{ass}
Next we recall some preliminaries of Sobolev spaces.
For $m = 0, 1$, we define the space $H^m \triangleq \{ \phi:
\partial_z^\alpha \phi \in L^2 ( {\mathbb R}^n ), \ \mbox {for any}
\ \alpha: =( \alpha_1, \cdots, \alpha_n ) \ \mbox {with} \ |\alpha|
:= | \alpha_1 | + \cdots + | \alpha_n | \leq m \}$ with the norm
\begin{eqnarray*}
\| \phi \|_m \triangleq \left \{ \sum_{ |\alpha| \leq m } \int_{{\mathbb R}^d}
| \partial_z^\alpha \phi (z) |^2 d z \right \}^{\frac{1}{2}} .
\end{eqnarray*}
If we denote by $H^{-1}$ the dual space of $H^1$ and set $V = H^1$, $H= H^0$, $V^* = H^{-1}$, then $( V, H, V^* )$ is a Gelfand triple. We further need some assumptions on the coefficients.
\begin{ass}\label{ass:sobolev}
	There exists a constant $C$ such that
	\begin{itemize}
		\item $|\sigma(t,x,u)|\le C$, for any $(t,x,u) \in [0,T]\times \mathbb R^n \times U$;
		\item $\|b(t,\cdot,u)\|_H \le C$ and $\|\int_E g(t,e,\cdot,u)\nu(de)\|_H \le C$, for any $(t,u) \in [0,T] \times  U$;
		\item $f(\cdot,\cdot,0,0,0) \in M_{\mathscr F}^2 ( 0, T; H )$.
	\end{itemize}
\end{ass}
With these assumptions, we can apply our previous result about the BSEEJ, which leads to the following theorem.
\begin{thm}\label{thm:HJB}
Let Assumptions \ref{ass:2.1}-\ref{ass:c4} and \ref{ass:6.2}-\ref{ass:sobolev} be satisfied.
Then the stochastic HJB equation \eqref{eq:6.17}
has a unique weak solution $(V, \Phi, \Psi) \in M_{\mathscr F}^2 ( 0, T; V )\times M_{\mathscr F}^2 ( 0, T; H )\times
M_{\mathscr F}^{\nu,2 }( 0, T; H)$ in the sense of Definition \ref{defn:6.2}.
\end{thm}

\begin{proof}
The proof is conducted by recasting the stochastic HJB equation \eqref{eq:6.17} as
the form   of the backward stochastic evolution equation. Under Assumptions \ref{ass:2.1}-\ref{ass:2.2}
and \ref{ass:6.2}, we define the mappings
$A: [ 0, T ] \times \Omega \rightarrow {\mathscr L} ( V, V^* ),$
$B: [ 0, T ] \times \Omega \rightarrow {\mathscr L} ( H, V^* ),$
$f: [0, T] \times \Omega \times V \times H \times V \times H \rightarrow H$, and $\xi: \Omega \rightarrow H$ by
\begin{eqnarray}\label{eq:6.44}
   &\langle A(t)w,\varphi\rangle
   =\frac{1}{2}\int_{R^n}\big\langle \sigma^\top(t,x)Dw(t,x), \sigma^\top (t, x) D\varphi(t, x)\big\rangle dx, \quad \forall \varphi,w\in V , \\
   &\langle B(t)\phi,\varphi\rangle
   =\int_{R^n} \langle \sigma_d(t,x)\varphi(t,x), D\varphi(t, x) \rangle dx,  \quad \forall \varphi\in V, \phi\in H ,
\end{eqnarray}
and
\begin{eqnarray*}
  F(t,w,\phi,\psi) &=& - \langle  \nabla \cdot[\sigma\sigma^\top(t,x)],  Dw(t,x) \rangle-\phi(t,x) \nabla \cdot [\sigma_d(t,x)]+\displaystyle+\inf_{u\in
  	U}\bigg [\langle  b(t,x,u), Dw(t,x)\rangle \\
  &&+f(t,x,u,V,\sigma^TDw+\phi,\int_E\left( \mathcal I w(t,e,x,u)+\psi(t,e,x+g)\right)l(t,e)\nu(de)) \\
  &&+\displaystyle\int_
  {E}\big[\mathcal I w(t,e,x,u)-(g(t, e,x,u), D w(t,x))\big]\nu(d e)+\displaystyle\int_{E}\displaystyle \mathcal I\psi(t,e,x,u)\nu(d e)\bigg],
\end{eqnarray*}
$\forall w\in V, \phi \in H, \psi\in M^{\nu,2}(H)$.

Using the above operators, we can rewrite the stochastic HJB equation \eqref{eq:6.17} as the following BSEEJ:
\begin{eqnarray}\label{eq:2.7}
\left\{
\begin{aligned}
d Y(t) =& \ \{ A (t) Y (t) +B(t)Z(t)+ F ( t, Y (t), Z (t), R(t, \cdot) ) \} d t + Z (t) d W_d (t)+ \int_E R(t,e)\tilde \mu(de,dt) ,  \ t \in [0, T] , \\
Y(T) =& \ \xi .
\end{aligned}
\right.
\end{eqnarray}
Then, \eqref{eq:6.48} can be written
as the following abstract formula:
\begin{eqnarray}\label{eq:c5}
( Y (t),  \phi )_H &=& (\xi, \phi)_H
- \int_t^T \Big\langle  A (s) Y (s) +B(s)Z(s)+F ( s, Y (s), Z (s), R(s, \cdot) ), \phi \Big\rangle d t \nonumber \\
&& - \int_t^T ( Z (s), \phi )_H d W (s)
 -\int_t^T\int_E (R(s,e), \phi)_H\tilde \mu(de,ds), \quad t \in [0, T] .
\end{eqnarray}
Therefore, we can apply the results in Subsection 6.1 to discuss the solvability of the stochastic HJB equation \eqref{eq:c41}.
In order to obtain the existence and uniqueness result for the backward stochastic evolution equation \eqref{eq:c46} by Theorem \ref{thm:6.1},
we need to check that Assumption \ref{ass:6.1} is satisfied. Indeed, Assumptions \ref{ass:6.1}(i) and \ref{ass:6.1}(iii) follow directly from Assumptions \ref{ass:2.1}-\ref{ass:2.2} and the definition
of mappings $A(t)$ and $F(t, w, \varphi,\psi)$.
Moreover, Assumption \ref{ass:6.1}(ii) on the coercivity of $A(t)$
is guaranteed by Assumption \ref{ass:6.2}.

From Assumption \ref{ass:6.2}, we have that for any $\phi\in H^1$,
\begin{eqnarray*}
  \langle A\phi, \phi\rangle- ||B^* \phi||^2_H
  &=&\frac{1}{2}\int_{R^n}\langle \sigma^\top(t,x)D\phi(t,x), \sigma^\top (t, x) D\phi(x)\rangle dx-\int_{R^n}|\sigma^\top_d(t,x)D\phi(t,x)|^2 dx \nonumber \\
  &=& \frac{1}{2}\int_{R^n}\langle  \hat \sigma\hat\sigma^\top(t,x)D\phi(t,x),D\phi(x)\rangle dx \nonumber \\
  &\geq& \alpha\int_{R^n}\langle D\phi(t,x),D\phi(t,x)\rangle dx \nonumber \\
  &=& \alpha||\phi||_V-\alpha||\phi||_H .
\end{eqnarray*}

Now it remains to check Assumption \ref{ass:6.1}(vi), i.e., the Lipschitz condition of the mapping $F$, is satisfied.
For notational simplicity, we denote by
\begin{eqnarray*}
J_1(t,w, \phi)&:=& -\langle  \nabla \cdot[\sigma\sigma^\top(t,\cdot)],  Dw(\cdot) \rangle-\phi(\cdot) \nabla \cdot [\sigma_d(t,\cdot)], \\
J_2(t,u,w)&:=&\langle  b(t,\cdot,u), Dw(\cdot)\rangle-\int_{E}\langle g(t, e,\cdot,u), D w(\cdot)\rangle\nu(d e) , \\
J_3(t,u, w)&:=&\int_{E}[ w( \cdot +g(t, e,\cdot ,u))- w(\cdot)]\nu(d e), \\
J_4(t,u,\psi)&:=& \int_{E}[\psi(e, \cdot+g(t, e,\cdot,u))-\psi(e, \cdot )]\nu(d e),\\
f(t,u,w,\phi,\psi)&:=&f(t,\cdot,u,\sigma^\top Dw+\phi,\int_E(\mathcal Iw(e,\cdot,u)+\psi(e,\cdot+g))l(t,e)\nu(de)).
\end{eqnarray*}
Using the above notations, we rewrite $F$ as
\begin{eqnarray}\label{eq:c46}
F(t,w,\phi, \psi)=J_1(t,w, \phi)+\inf_{u\in
U}\big(J_2(t,u,w)+J_3(t,u,w)+J_4(t,u,\psi)+f(t,u,w,\phi,\psi) \big) .
\end{eqnarray}
By the boundedness property of $b$, $f$, $\pi$, $\sigma$
and their derivatives, we see that there exists a
positive constant $C$ such that for any $w_1, w_2 \in V$, $\phi_1, \phi_2 \in
H, u\in U$ and a.e. $( t, \omega ) \in [ 0, T ]\times \Omega$,
\begin{eqnarray}\label{eq:6.51}
||J_1(t, w_1,\phi_1)-J_1(t, w_2, \phi_1)||_H\leq C ||w_1-w_2||_V+||\phi_1-\phi_2||_H,
\end{eqnarray}
and
\begin{eqnarray}\label{eq:6.52}
||J_2(t, u, w_1)-J_2(t, u,w_2)||_H\leq C ||w_1-w_2||_V.
\end{eqnarray}
Moreover, using the variable transformation and Assumption
\ref{ass:c4}, for any $w_1, w_2 \in V$, we have
\begin{eqnarray}\label{eq:6.53}
&& ||J_3 (t,u,w_1)-J_3(t,u,w_2) ||_H^2 \nonumber \\
&& = \int_{\mathbb R^n}\bigg | \int_{E}\big[ (w_1(x+g(t, e, x,u))-w_1(x))-
    (w_2(x+g(t, e,x,u))-w_2(x)) \big]\nu(d e)\bigg|^2dx \nonumber \\
&& = \int_{\mathbb R^n}\bigg |\int_{E}\big[(w_1(x+g(t, e, x,u))
    -w_2(x+g(t, e,x,u)))+(w_1(x)-w_2(x))\big]\nu(d e)\bigg|^2 dx \nonumber \\
&& \leq 2\nu (E)\int_
    {E}\int_{\mathbb R^n}\bigg|w_1(x+g(t, e,
    x,u))-w_2(x+g(t, e,x,u))\bigg|^2dx\nu (d e)
    +2\nu (E)\int_
    {E}\int_{\mathbb R^n}\bigg|w_1(x)
    -w_2(y)\bigg|^2dx \nu(d e) \nonumber \\
&& = 2\nu(E)\int_{E}\int_{R^n}|w_1(y)
    -w_2(y)|^2|\mbox{det}(I+\partial_xg(t, e,x,u)|^{-1}dy\nu(d e)
    +2\nu^2({E})||w_1-w_2||_H^2 \nonumber \\
&& \leq 2\nu^2({E})(1+{\delta}^{-1})||w_1-w_2||_H^2 \nonumber \\
&& \leq 2\nu^2({E})(1+{\delta}^{-1})||w_1-w_2\|_V^2.
\end{eqnarray}
Similarly, using variable transformation and Assumption
\ref{ass:c4}, we can easily obain
\begin{eqnarray}\label{eq:6.54}
||J_4(t,u,\psi_1)-J_4(t,u, \psi_2) ||_H^2\leq
 2\nu({E})(1+{\delta}^{-1})||\psi^1(x)-\psi^2(x)||_{M^{\nu,2}(E;H)}^2, \quad \forall (t,u)\in[0, T]\times U.
\end{eqnarray}
From previous result and Lipschitz continuity of $f$, we also have that
\begin{equation}\label{eq:6.55}
\|f(t,u,w_1,\phi_1,\psi_1)-f(t,u,w_2,\phi_2,\psi_2)\|_H^2\le C(\|w_1-w_2\|^2_V+\|\phi_1-\phi_2\|_H^2+\|\psi_1-\psi_2\|^2_{M^{\nu,2}(E;H)}).
\end{equation}
Finally, as in \cite{Peng92a}, the Lipschitz condition on the mapping $F$ can be easily derived from the above
inequalities (\ref{eq:6.51})-(\ref{eq:6.55}). Consequently, an application of Theorem \ref{thm:6.1} shows that the stochastic HJB equation
\eqref{eq:c41} has a unique solution.
\end{proof}
\begin{rmk}
The above result can be extended to a  more general case, where the randomness of the coefficients comes from part, but not all, of the Brownian motions. More precisely, let $d_1,d_2$ be two integers such that $d_1 \ge 1$ and $d_1+d_2=d$.  Assume that all the coefficients are predictable with respect to the filtration generated by the Brownian motions $(W_{d_1+1},\cdots,W_{d})$ and the Poisson random measure $\tilde \mu$.  We further assume $\tilde\sigma(\cdot):=(\sigma_1(\cdot),\sigma_2(\cdot),\cdots,\sigma_{d_1}(\cdot))$ satisfies the non-degenerate assumption \eqref{ieq:non-degen}. Then, the corresponding backward HJB equation admits a Sobolev solution. The proof for such a result is almost the same as that of Theorem \ref{thm:HJB} with a minor modification. Thus, we do not repeat it here.
\end{rmk}

\section{Conclusion}

In this paper, we study the stochastic HJB equation with random coefficients and jumps. We prove that the value function is a solution for the stochastic HJB equation if some regularity assumptions are satisfied. The idea of the proof is motivated by the method used by Tang \cite{tang2015dynamic} which studied the Riccati equation for the stochastic LQ problem. The stochastic Riccati equation is a special example of the stochastic HJB equation. However, we have to mention that our result do not include the LQ problem as a special case since some assumptions, like compact control region and linear growth generator $f$, do not hold in that case. These technical assumptions ensure that our proof is rigorous. In Zhang et al. \cite{zhang2020backward}, the authors proved the solvability of backward stochastic Riccati equation with random jumps. Their basic idea is similar to ours, but the proof heavily relies on the linear-quadratic structure of their problem. Thus, the result of Zhang et al \cite{zhang2020backward} is stronger than ours in the sense that the semi-martingale structure of value function is proved in \cite{zhang2020backward}, while we make it as an assumption. To our best knowledge, the question that under which conditions the value function is a semi-martingale remains an open problem. We also consider the stochastic HJB equation in the Sobolev space under some non-degenerate assumption which is standard in the study of BSPDEs. For the most general case, a classical or Sobolev solution is still an open problem even for deterministic case. People often consider the so-called viscosity solution. For BSPDE case, there are just few paper on this subject. The readers are referred to \cite{qiu2018viscosity} for more discussion.

\bibliographystyle{plain}

\end{document}